\documentclass[10pt,reqno]{amsart}
\usepackage{amsmath,amsthm,amssymb}
\usepackage{nccmath}
\usepackage{a4wide}
\usepackage[numbers]{natbib}
\usepackage[latin1]{inputenc}
\usepackage{latexsym}
\usepackage{amssymb}
\usepackage[mathscr]{eucal}
\usepackage{epsfig}
\usepackage{lscape}
\usepackage{amssymb}
\usepackage{amsmath}
\usepackage{amsfonts}
\usepackage{graphicx}
\usepackage{xcolor}
\usepackage{yhmath}
\usepackage{mathdots}
\usepackage{MnSymbol}
\usepackage{epsfig}
\usepackage{lineno}
\usepackage{tikz}
\usepackage{bbold}
\usepackage[absolute]{textpos}

\newcommand{\NN}{\mathbb{N}}
\newcommand{\ZZ}{\mathbb{Z}}
\newcommand{\RR}{\mathbb{R}}
\newcommand{\CC}{\mathbb{C}}

\newcommand{\Span}{\mathrm{span}\,}

\newcommand{\rank}{\mathrm{rank\,}}

\newlength{\wdth}

\newtheorem{theorem}{Theorem}[section]
\newtheorem{lemma}[theorem]{Lemma}
\newtheorem{corollary}[theorem]{Corollary}

\newtheorem{proposition}[theorem]{Proposition}

\newtheorem{definition}[theorem]{Definition}
\newtheorem{remark}[theorem]{Remark}
\usepackage{mathtools}  
\usepackage{mathrsfs}   

\def\ps@pprintTitle{%
     \let\@oddhead\@empty
     \let\@evenhead\@empty
     \let\@evenfoot\@oddfoot}

\title{Orthogonal Laurent polynomials of two real variables}

\author{Ruymán Cruz-Barroso and Lidia Fern\'andez}

\address{Ruym\'an Cruz-Barroso\\
Department of Mathematical Analysis and Instituto de Matem\'aticas y sus Aplicaciones (IMAULL) \\
La Laguna University \\
38271 La Laguna,Tenerife. Canary Islands. Spain.}
\email{rcruzb@ull.es}

\address{Lidia Fern\'andez\\
IMAG and Departamento de Matem\'atica Aplicada\\
Universidad de Granada\\
18071, Granada, Spain.}
\email{lidiafr@ugr.es}

\thanks{The work of the first author was partially supported by IMAG-Mar\'ia de Maeztu grant CEX2020-001105-M.
}

\date{\today}

\subjclass[2010]{33C45, 42C05}

\keywords{Orthogonal Laurent polynomials of two real variables, balanced ordering, recurrence relations, Christoffel-Darboux and confluent formula, Favard's theorem.}

\begin{document}

\maketitle

\begin{abstract}
In this paper we consider an appropriate ordering of the Laurent monomials $x^{i}y^{j}$, $i,j \in \ZZ$ that allows us to study sequences of orthogonal Laurent polynomials of the real variables $x$ and $y$ with respect to a positive Borel measure $\mu$ defined on $\RR^2$ such that $\{ x=0 \}\cup \{ y=0 \} \not\in \textrm{supp}(\mu)$. This ordering is suitable for considering the {\em multiplication plus inverse multiplication operator} on each varibale $\left( x+\frac{1}{x}\right.$ and $\left. y+\frac{1}{y}\right)$, and as a result we obtain five-term recurrence relations, Christoffel-Darboux and confluent formulas for the reproducing kernel and a related Favard's theorem. A connection with the one variable case is also presented, along with some applications for future research.
\end{abstract}


\section{Introduction}\label{secIntro}

Orthogonal Laurent polynomials with respect to a positive Borel measure supported on the real line were introduced for the first time in \cite{JT1}, and also implicitly in \cite{JTW} in relation to continued fractions and the solution of the strong Stieltjes moment problem (see also chronologically \cite{OT1,JNT1,OT2,SCop,JO1}). However, it is convenient to say that they can be viewed as orthogonal polynomials with respect to varying weights, which have been extensively used in the theory of constructive approximation (see for example \cite{HvR} where they appear in the context of continuous T-fractions). The situation when the measure is supported on the unit circle is very much simplified, see e.g. \cite{RCB1}. An extensive bibliography has been produced after these works, giving rise to a theory close to the well known theory of orthogonal polynomials on the real line (see e.g. \cite{But,CDM1,CDM2,RCB5,CDM3,RCB3}), with applications in moment problems, recurrence relations, reproducing kernels, Favard's theorem, interpolation and quadrature formulas along with denseness and convergence, linear algebra and inverse eigenvalue problems, Krylov methods, model reduction, linear prediction, system identification,$\ldots$. This theory has been also considered for positive Borel measures supported on the unit circle (for the first time \cite{Thr}), giving rise in particular to the well known CMV theory (\cite{CMV1}, see also e.g. \cite{RCB1,RCB2}) that has produced an important impulse in the theory of orthogonal polynomials on the unit circle (see \cite{Simon}). In particular, in the context of quadrature formulas on the real line, the advantages of considering rules based on Laurent polynomials instead of ordinary polynomials have been shown deeply in the literature, theoretically and numerically. For example, from a theoretical point of view, the theory of Orthogonal Laurent Polynomials has allowed the development of the Theory of Strong Moment Problems (see Subsection \ref{Appl2} further). From a numerical point of view, a comparison between quadrature formulas on the real half-line based on orthogonal polynomials and orthogonal Laurent polynomials is carried out in \cite{RCB5}, and it is shown there that when the integrand presents singularities near the subset of integration, the results of the classical Gaussian quadrature rules are strongly improved when quadrature formulas based on Laurent polynomials are used.

In the one variable case, there are few weight functions on the real line that give rise to explicit expressions for the corresponding orthogonal Laurent polynomials. In practice these orthogonal Laurent polynomials are computed recursively from a three-term recurrence relation that holds for an arbitrarily ordered sequence of monomials $\{ x^{i} \}_{i \in \ZZ}$ (induced from what is known in the literature as ``generating sequence"). In the particular case of the ``balanced" ordering
\begin{equation}\label{balanced1v}
\mathscr{L}=\Span \left\{ 1, x, \frac{1}{x}, x^2, \frac{1}{x^2}, \ldots \right\},
\end{equation} this recurrence is given by the following (see \cite{CDM2,CDM3})
\begin{theorem}\label{recu1var}
Let $\omega$ be a positive Borel measure on $\RR^{+}$ and let $\{ \psi_k \}_{k \geq 0}$ be the sequence of orthonormal Laurent polynomials induced by the inner product $\langle f,g \rangle_{\omega}=\int_0^{\infty} f(x)g(x)d\omega(x)$ and the balanced ordering (\ref{balanced1v}). Then, setting $\psi_{-1} \equiv 0$, there exist two sequences of positive real numbers $\{ \Omega_n \}_{n \geq 0}$ and $\{ C_n \}_{n \geq 0}$ such that for all $n\geq 0$,
$$\begin{array}{cclcl}C_n\psi_{n+1}(x)&=&\left( \Omega_nx-1 \right)\psi_n(x)-C_{n-1}\psi_{n-1}(x) &\quad& \textrm{if n is even}, \\
C_n\psi_{n+1}(x)&=&\left( 1-\frac{\Omega_n}{x} \right)\psi_n(x)-C_{n-1}\psi_{n-1}(x) &\quad& \textrm{if n is odd}. \end{array}$$
Furthermore, $\psi_0 \equiv \frac{1}{\sqrt{m_0}}$, $\Omega_0=\frac{m_0}{m_1}$ and $C_0=\frac{\sqrt{m_2m_0-m_1^2}}{m_1}$, $m_k=\langle x^k,1 \rangle$ being corresponding moments for $\omega$, $k\in \ZZ$.
\end{theorem}

On the other hand, the general theory of multivariate orthogonal polynomials is still far from being considered an established field and has experienced delayed development, especially in fundamental aspects. In 1865, C. Hermite \cite{Her65} explored a two-variable generalization of Legendre polynomials, marking the initial appearance of orthogonal polynomial families in multiple variables in the literature. However, it was not until 1926 that a study on families of orthogonal polynomials in two variables on the unit disk and the triangle appeared in the classic monograph by Appell and Kampé de Fériet \cite{AKF26}. Since that moment, various authors have contributed to the development of the general theory of polynomials in several variables; see, for example, \cite{Ko75, KS67, Su99}.

Based on a vectorial representation, M. Kowalski (\cite{Ko82a, Ko82b}) proposed a novel approach in the study of polynomials in multiple variables. This perspective has allowed the development of a basic algebraic theory, which can be found in the monograph by C. F. Dunkl and Y. Xu (\cite{Xu}). In particular, it has been possible to extend fundamental properties to multiple variables, such as the three-term relation, Favard's theorem or the Christoffel-Darboux formula. The monograph \cite{Xu} comes highly recommended as reference for gaining insight into the current state of the art in multivariate orthogonal polynomials.

Orthogonal polynomials in several variables find diverse applications across fields like physics, quantum mechanics, and signal processing. One prominent application lies in optics and ophthalmology. Zernike polynomials are orthogonal polynomials on the unit disk \cite{Zer34} and were introduced by Fritz Zernike (Nobel prize for physics in 1953) in 1934 to address optical challenges related to telescopes and microscopes. In the year 2000, the Optical Society of America adopted them as the standard pattern in ophthalmic optics.

The purpose of this paper is to consider for the first time in the literature (as far as we know) the theory of sequences of orthogonal Laurent polynomials in several real variables. The advantages of considering orthogonal Laurent polynomials (or more generally, orthogonal rational functions) instead of ordinary orthogonal polynomials have been shown in a wide variety of contexts in the literature of the one variable case. The growing interest in the study of orthogonal polynomials in several variables undoubtedly motivates to consider generalizations to more general kind of functions than ordinary polynomials, mainly due to their possible applications in many problems like cubature rules, Fourier orthogonal series and summability of orthogonal expansions, moment problems,$\ldots$

In \cite{AM22}, multivariate orthogonal Laurent polynomials in the unit torus are studied. The authors provide an ordering for the monomials, focusing specifically on the moment matrix, Fourier series, Christoffel-Darboux formulas and related concepts. They work with complex variables and they prove three-term relations but they do not find a Favard's type theorem. The study explores also Christoffel-type perturbations of the measure through multiplication by Laurent polynomials. Both discrete and continuous deformations of the measure result in a Toda-type integrable hierarchy. The variables in this paper are real and the ordering considered is not the same that the one used in \cite{AM22} so the results obtained here are not included in it.

For simplicity, we will restrict to the case of two real variables, but all the results can be extended to more variables by using a somewhat more involved notation. Here, the basic key is to start from an appropriate ordering for the Laurent monomials $x^{i}y^{j}$, $i,j \in \ZZ$ that is inspired on the ``balanced case" (that is usually considered in the literature), but now for both real variables simultaneously. The vectorial representation of the Laurent polynomials is necessary for the proof of the main results.

The paper has been organized as follows. An appropriate ordering of the Laurent monomials $x^{i}y^{j}$, $i,j \in \ZZ$ for the construction of Laurent polynomials sequences of two real variables with respect to a linear functional is considered in Section 2. We concentrate in the positive-definite case, dealing with orthogonality with respect to a positive Borel measure $\mu$ defined on $\RR^2$ such that $\{ x=0 \}\cup \{ y=0 \} \not\in \textrm{supp}(\mu)$. Five-term recurrence relations are obtained involving multiplication by $x+\frac{1}{x}$ and $y+\frac{1}{y}$. In Section 3 we deduce a Favard's theorem and Christoffel-Darboux and confluent formulas for the reproducing kernel, whereas in Section 4 we present a connection with the one variable case when $\mu$ is supported in a rectangle and it is of the form $d\mu(x,y)=d\mu_1(x)d\mu_2(y)$. Some conclusions are finally carried out.

We end this introduction with some remaining notation throughout the paper. We denote by $E[\cdot]$ the integer part function, by $\delta_{k,l}$ the Kronecker delta symbol, by ${\mathcal M}_{n,m}$ the space of (real) matrices of dimension $n \times m$, ${\mathcal M}_{n}$ being the space of square (real) matrices of dimension $n$, by ${\mathcal I}_n$ the identity matrix of dimension $n$, by ${\mathcal O}_{n,m}$ and ${\mathcal O}_{n}$ the zero matrices in ${\mathcal M}_{n,m}$ and ${\mathcal M}_n$, respectively, and by $\textrm{diag}(a_1,\ldots,a_n)\in {\mathcal M}_n$ the diagonal matrix with ordered entries in the main diagonal $a_1,\ldots,a_n$.

\section{Orthogonal Laurent polynomials of two real variables. Five-term relations}\label{SecOLP2VRR}

In the one-variable situation it is usual to consider a nested sequence of subspaces of Laurent polynomials $\left\{ \mathscr{L}_n \right\}$ such that $\mathscr{L}_0=\Span \{ 1 \}$, $\mathscr{L}_n \subset \mathscr{L}_{n+1}$, $\textrm{dim} \left( \mathscr{L}_{n} \right)=n+1$ for all $n \geq 0$, and $\bigcup_{n \geq 0} \mathscr{L}_n = \mathscr{L}$. See e.g. \cite{CDM2,CDM4}.

Having in mind the ``balanced" ordering (\ref{balanced1v}) in the one variable situation
$$\mathscr{L}_0=\Span \{ 1 \}, \quad \mathscr{L}_{2k}=\Span \left\{ \frac{1}{x^k},\ldots,x^k \right\}, \quad \mathscr{L}_{2k-1}=\mathscr{L}_{2k-2} \oplus \Span \{ x^{k} \}, \quad \forall k\geq 1,$$
(see e.g. \cite{SCop,RCB5,CDM3,JNT1} for the real line case, and \cite{Thr,Simon,CMV1,RCB2} for the unit circle case), we can proceed by defining the sequence
\begin{equation}\label{defc}
c_n=(-1)^{n+1} \cdot E \left[ \frac{n+1}{2} \right], \quad \forall n \geq 0
\end{equation}
and considering the Laurent monomials
$$p_{m,n}(x,y)=x^{c_m}y^{c_n}, \quad \forall m,n \geq 0$$
and the infinite matrix
\begin{equation}\label{matrixL}
\begin{array}{cccccccc}
p_{0,0}=x^{c_0}y^{c_0}=1 & \; & p_{1,0}=x^{c_1}y^{c_0}=x & \; & p_{2,0}=x^{c_2}y^{c_0}=\frac{1}{x} & \; & p_{3,0}=x^{c_3}y^{c_0}=x^2 & \cdots
\\ \\
p_{0,1}=x^{c_0}y^{c_1}=y & \; & p_{1,1}=x^{c_1}y^{c_1}=xy & \; & p_{2,1}=x^{c_2}y^{c_1}=\frac{y}{x} & \; & p_{3,1}=x^{c_3}y^{c_1}=x^2y & \cdots
\\ \\
p_{0,2}=x^{c_0}y^{c_2}=\frac{1}{y} & \; & p_{1,2}=x^{c_1}y^{c_2}=\frac{x}{y} & \; & p_{2,2}=x^{c_2}y^{c_2}=\frac{1}{xy} & \; & p_{3,2}=x^{c_3}y^{c_2}=\frac{x^2}{y} & \cdots \\ \vdots & & \vdots & & \vdots & & \vdots & \ddots
\end{array}
\end{equation}

Setting ${\mathcal L}=\Span \{ x^{i}y^{j} : i,j\in \ZZ \}$, the space of Laurent polynomials of real variables $x$ and $y$, we can order these elements $p_{n,m}$ by anti-diagonals in (\ref{matrixL}) as
\begin{equation}\label{L} {\mathcal L} = \Span \{ \underbrace{p_{0,0}}_{n+m=0}, \underbrace{p_{1,0}, p_{0,1}}_{n+m=1}, \underbrace{p_{2,0}, p_{1,1}, p_{0,2}}_{n+m=2}, \underbrace{p_{3,0}, p_{2,1}, p_{1,2}, p_{0,3}}_{n+m=3}, \underbrace{p_{4,0}, p_{3,1}, p_{2,2}, p_{1,3}, p_{0,4}}_{n+m=4}, \cdots \; \}\end{equation}
and define
\begin{equation}\label{Ln}
{\mathcal L}_n = \Span \left\{ p_{i,j} : i+j \leq n \right\}, \quad\textrm{for all} \;n\geq 0, \quad \quad \textrm{dim}\left( {\mathcal L}_n \right)=\frac{(n+1)(n+2)}{2}, \quad \quad {\mathcal L}=\bigcup_{n \geq 0} {\mathcal L}_n.
\end{equation}

Consider
\begin{equation}\label{phik}\phi_k(x,y)=\left( \begin{array}{c} p_{k,0}(x,y) \\ p_{k-1,1}(x,y) \\ \vdots \\ p_{0,k}(x,y) \end{array} \right) \in {\mathcal M}_{k+1,1}, \quad \forall \;k \geq 0,
\end{equation} that is, the components of the vector $\phi_k$ are the $k+1$ linearly independent Laurent monomials of ${\mathcal L}_k \backslash {\mathcal L}_{k-1}$, ordered as they appear in the expansion (\ref{L}). So, we can interpret ${\mathcal L}_{n} = \Span \{ \phi_0, \ldots,\phi_n \}$, for all $n\geq 0$ so that
if $\psi_k \in {\mathcal L}_{k}$ for some $k \geq 0$, then $\psi_k = \sum_{l=0}^{k} C_l\phi_l$ where $C_l \in {\mathcal M}_{1,l+1}$ are constant matrices ($C_k$ being the leading coefficient matrix).

Observe that $\phi_0 \equiv 1$ and for all $l \geq 1$,
$$\phi_{2l}(x,y)=\left( \begin{array}{c} x^{-l}y^0 \\ x^ly \\ x^{-(l-1)}y^{-1} \\ \vdots \\ xy^l \\ x^0y^{-l} \end{array} \right), \quad \quad \phi_{2l-1}(x,y)=\left( \begin{array}{c} x^{l}y^0 \\ x^{-(l-1)}y \\ x^{l-1}y^{-1} \\ \vdots \\ xy^{-(l-1)} \\ x^0y^{l} \end{array} \right).$$
Thus,
\begin{equation}\label{mult1}
\begin{array}{ccc}
x\phi_{2l}(x,y)= \left( \begin{array}{c} p_{2l-2,0} \in {\mathcal L}_{2l-2} \backslash {\mathcal L}_{2l-3} \\ p_{2l+1,1} \in {\mathcal L}_{2l+2} \backslash {\mathcal L}_{2l+1}  \\ p_{2l-4,2} \in {\mathcal L}_{2l-2} \backslash {\mathcal L}_{2l-3}  \\ \vdots \\ p_{0,2l-2} \in {\mathcal L}_{2l-2} \backslash {\mathcal L}_{2l-3}  \\ p_{3,2l-1} \in {\mathcal L}_{2l+2} \backslash {\mathcal L}_{2l+1}  \\ p_{1,2l} \in {\mathcal L}_{2l+1} \backslash {\mathcal L}_{2l}\end{array} \right), &\quad &
\frac{1}{x}\phi_{2l}(x,y)=\left( \begin{array}{c} p_{2l+2,0} \in {\mathcal L}_{2l+2} \backslash {\mathcal L}_{2l+1} \\ p_{2l-3,1} \in {\mathcal L}_{2l-2} \backslash {\mathcal L}_{2l-3}  \\ p_{2l,2} \in {\mathcal L}_{2l+2} \backslash {\mathcal L}_{2l+1}  \\ \vdots \\ p_{4,2l-2} \in {\mathcal L}_{2l+2} \backslash {\mathcal L}_{2l+1} \\ p_{0,2l-1} \in {\mathcal L}_{2l-1} \backslash {\mathcal L}_{2l-2} \\ p_{2,2l} \in {\mathcal L}_{2l+2} \backslash {\mathcal L}_{2l+1}\end{array} \right),
\end{array}\end{equation}

\begin{equation}\label{mult2}\begin{array}{ccc}
x\phi_{2l-1}(x,y)= \left( \begin{array}{c} p_{2l+1,0} \in {\mathcal L}_{2l+1} \backslash {\mathcal L}_{2l} \\ p_{2l-4,1} \in {\mathcal L}_{2l-3} \backslash {\mathcal L}_{2l-4}  \\ p_{2l-1,2} \in {\mathcal L}_{2l+1} \backslash {\mathcal L}_{2l}  \\ \vdots \\ p_{0,2l-3} \in {\mathcal L}_{2l-3} \backslash {\mathcal L}_{2l-4}  \\ p_{3,2l-2} \in {\mathcal L}_{2l+1} \backslash {\mathcal L}_{2l}  \\ p_{1,2l-1} \in {\mathcal L}_{2l} \backslash {\mathcal L}_{2l-1}\end{array} \right), &\quad &
\frac{1}{x}\phi_{2l-1}(x,y)= \left( \begin{array}{c} p_{2l-3,0} \in {\mathcal L}_{2l-3} \backslash {\mathcal L}_{2l-4} \\ p_{2l,1} \in {\mathcal L}_{2l+1} \backslash {\mathcal L}_{2l}  \\ p_{2l-5,2} \in {\mathcal L}_{2l-3} \backslash {\mathcal L}_{2l-4}  \\ \vdots \\ p_{4,2l-3} \in {\mathcal L}_{2l+1} \backslash {\mathcal L}_{2l} \\ p_{0,2l-2} \in {\mathcal L}_{2l-2} \backslash {\mathcal L}_{2l-3} \\ p_{2,2l-1} \in {\mathcal L}_{2l+1} \backslash {\mathcal L}_{2l}\end{array} \right),
\end{array}\end{equation}

\begin{equation}\label{mult3}
\begin{array}{ccc}
y\phi_{2l}(x,y)= \left( \begin{array}{c} p_{2l,1} \in {\mathcal L}_{2l+1} \backslash {\mathcal L}_{2l} \\ p_{2l-1,3} \in {\mathcal L}_{2l+2} \backslash {\mathcal L}_{2l+1}  \\ p_{2l-2,0} \in {\mathcal L}_{2l-2} \backslash {\mathcal L}_{2l-3}  \\ \vdots \\ p_{2,2l-4} \in {\mathcal L}_{2l-2} \backslash {\mathcal L}_{2l-3}  \\ p_{1,2l+1} \in {\mathcal L}_{2l+2} \backslash {\mathcal L}_{2l+1}  \\ p_{0,2l-2} \in {\mathcal L}_{2l-2} \backslash {\mathcal L}_{2l-3}\end{array} \right), &\quad &
\frac{1}{y}\phi_{2l}(x,y)=\left( \begin{array}{c} p_{2l,2} \in {\mathcal L}_{2l+2} \backslash {\mathcal L}_{2l+1} \\ p_{2l-1,0} \in {\mathcal L}_{2l+1} \backslash {\mathcal L}_{2l}  \\ p_{2l-2,4} \in {\mathcal L}_{2l+2} \backslash {\mathcal L}_{2l+1}  \\ \vdots \\ p_{2,2l} \in {\mathcal L}_{2l+2} \backslash {\mathcal L}_{2l+1} \\ p_{1,2l-3} \in {\mathcal L}_{2l-2} \backslash {\mathcal L}_{2l-3} \\ p_{0,2l+2} \in {\mathcal L}_{2l+2} \backslash {\mathcal L}_{2l+1}\end{array} \right),
\end{array}\end{equation}
and
\begin{equation}\label{mult4}\begin{array}{ccc}
y\phi_{2l-1}(x,y)= \left( \begin{array}{c} p_{2l-1,1} \in {\mathcal L}_{2l} \backslash {\mathcal L}_{2l-1} \\ p_{2l-2,3} \in {\mathcal L}_{2l+1} \backslash {\mathcal L}_{2l}\\ p_{2l-3,0} \in {\mathcal L}_{2l-3} \backslash {\mathcal L}_{2l-4}  \\ \vdots \\ p_{2,2l-1} \in {\mathcal L}_{2l+1} \backslash {\mathcal L}_{2l}\\ p_{1,2l-4} \in {\mathcal L}_{2l-3} \backslash {\mathcal L}_{2l-4}  \\ p_{0,2l+1} \in {\mathcal L}_{2l+1} \backslash {\mathcal L}_{2l}\end{array} \right), &\quad &
\frac{1}{y}\phi_{2l-1}(x,y)= \left( \begin{array}{c} p_{2l-1,2} \in {\mathcal L}_{2l+1} \backslash {\mathcal L}_{2l} \\ p_{2l-2,0} \in {\mathcal L}_{2l-2} \backslash {\mathcal L}_{2l-3}  \\ p_{2l-3,4} \in {\mathcal L}_{2l+1} \backslash {\mathcal L}_{2l}  \\ \vdots \\ p_{2,2l-5} \in {\mathcal L}_{2l-3} \backslash {\mathcal L}_{2l-4} \\ p_{1,2l} \in {\mathcal L}_{2l+1} \backslash {\mathcal L}_{2l} \\ p_{0,2l-3} \in {\mathcal L}_{2l-3} \backslash {\mathcal L}_{2l-4}\end{array} \right).
\end{array}\end{equation}

It follows from (\ref{mult1})-(\ref{mult2}) and (\ref{mult3})-(\ref{mult4}) that if $\psi_k \in {\mathcal L}_k \backslash {\mathcal L}_{k-1}$, then $x\psi_k \in {\mathcal L}_{k+2}$, $\frac{1}{x}\psi_k \in {\mathcal L}_{k+2}$ and $y\psi_k \in {\mathcal L}_{k+2}$, $\frac{1}{y}\psi_k \in {\mathcal L}_{k+2}$, respectively. However, the key fact in what follows is that all the components of the vectors $\left(x+\frac{1}{x}\right)\psi_k$ and $\left(y+\frac{1}{y}\right)\psi_k$ are in ${\mathcal L}_{k+2} \backslash {\mathcal L}_{k+1}$.

A \emph{Laurent system} in two variables $\{\varphi_n\}_{n\ge 0}$ is a sequence of vectors of increasing size
$$
\varphi_n \in {\mathcal M}_{n+1,1}, \quad \varphi_n \in {\mathcal L}_n \backslash {\mathcal L}_{n-1}, \quad \forall \;n \geq 0
$$
such that the components in the vector $\varphi_n$ are linearly independent. It is clear that in this case
\begin{equation}\label{varphisexpand}
\varphi_{n}=\sum_{i=0}^{n}A_{i}^{(n)}\phi_i,  \quad \text{ with } A_{i}^{(n)}\in {\mathcal M}_{n+1,i+1} \text{ constant matrices, } A_{n}^{(n)} \text{ being regular. }
\end{equation}

Let us consider a linear functional $L$ defined in ${\mathcal L}$ by $L(x^i y^j)=\mu_{i,j}$ for $i,j\in\mathbb{Z}$ and extended by linearity. It can be defined over product of vectors in the following way

\begin{equation}\label{innervector}
L(f\,  g^T)= \left(L(f_i \, g_j) \right)_{i=1,\ldots,k; \; j=1,\ldots,m} \in {\mathcal M}_{k,m}, \;\; \textrm{where} \; f=[f_1,\ldots,f_k]^T \;\; \textrm{and} \;\; g=[g_1,\ldots,g_m]^T.
\end{equation}

\begin{definition}
A Laurent system $\{\varphi_n\}_{n\ge 0}$ is a \emph{system of orthogonal Laurent polynomials} with respect to the linear functional $L$ if for all $n\ge 0$
\begin{equation}\label{orth_cond_1}
\begin{aligned}
&L(\varphi_n \varphi_k^T) = {\mathcal O}_{n+1,k+1}, \quad k=0,\dots,n-1 \\
&L(\varphi_n \varphi_n^T) = {\mathcal H}_{n}\in {\mathcal M}_{n+1} \quad \text{with } {\mathcal H}_{n} \text{ an invertible matrix.}
\end{aligned}
\end{equation}
In the case when ${\mathcal H}_{n}={\mathcal I}_{n+1}$ for all $n\ge 0$, $\{\varphi_n\}_{n\ge 0}$ is called a \emph{system of orthonormal Laurent polynomials}.
\end{definition}
Observe that the orthogonality conditions are equivalent to
\begin{equation}\label{orth_cond_2}
\begin{aligned}
&L(\phi_k \varphi_n ^T) = {\mathcal O}_{k+1,n+1}, \quad k=0,\dots,n-1 \\
&L(\phi_n \varphi_n^T) = {\mathcal S}_{n}\in {\mathcal M}_{n+1} \quad \text{with } {\mathcal S}_{n} \text{ an invertible matrix.}
\end{aligned}
\end{equation}

For $n\ge 0$, $k,l\ge 0$, we define the matrices
$$
\mathrm{M}_{k,l}=L( \phi_k \, \phi_l^T)
$$
and the matrix
$$
\mathrm{M}_n=\left(\mathrm{M}_{k,l}\right)_{k,l=0}^n \quad \text{with} \quad \Delta_n=\det \mathrm{M}_n.
$$
We call $\mathrm{M}_n$ a \emph{moment matrix}.
Observe that

$$
\Delta_0=\left|\mu_{0,0}\right|, \;\;
\Delta_1=\left|\begin{array}{c|cc}
\mu_{0,0} & \mu_{1,0} & \mu_{0,1} \\
\hline
\mu_{1,0} & \mu_{2,0} & \mu_{1,1} \\
\mu_{0,1} & \mu_{1,1} & \mu_{0,2}
\end{array}\right|, \;\;
\Delta_2=\left|\begin{array}{c|cc|ccc}
\mu_{0,0} & \mu_{1,0} & \mu_{0,1} & \mu_{-1,0} & \mu_{1,1} & \mu_{0,-1} \\
\hline
\mu_{1,0} & \mu_{2,0} & \mu_{1,1} & \mu_{0,0} & \mu_{2,1} & \mu_{1,-1} \\
\mu_{0,1} & \mu_{1,1} & \mu_{0,2} & \mu_{-1,1} & \mu_{1,2} & \mu_{0,0} \\
\hline
\mu_{-1,0} & \mu_{0,0} & \mu_{-1,1} & \mu_{-2,0} & \mu_{0,1} & \mu_{-1,-1} \\
\mu_{1,1} & \mu_{2,1} & \mu_{1,2} & \mu_{0,1} & \mu_{2,2} & \mu_{1,0} \\
\mu_{0,-1} & \mu_{1,-1} & \mu_{0,0} & \mu_{-1,-1} & \mu_{1,0} & \mu_{0,-2} \\
\end{array}\right|,\ldots
$$

\begin{proposition}
A system of orthogonal Laurent polynomials with respect to the linear functional $L$ exists, if and only if, $\Delta_n\ne 0$ for all $n\ge 0$.
\end{proposition}
\begin{proof}
Using that $\varphi_n=\sum_{i=0}^n A_i \phi_i$ we have
$$
L(\phi_k \varphi_n^T)=\sum_{k=0}^n L(\phi_k \phi_i^T) A_i^T=\sum_{k=0}^n \mathrm{M}_{k,i} A_i^T
$$
The orthogonality conditions \eqref{orth_cond_2} are equivalent to the following linear system of equations:
$$
\mathrm{M}_n \begin{pmatrix} A_0^T \\ \vdots \\ A_{n-1}^T \\ A_n^T \end{pmatrix}= \begin{pmatrix} 0 \\ \vdots \\ 0 \\ {\mathcal S}_{n} \end{pmatrix}
$$
The system has a unique solution if the matrix $\mathrm{M}_n$ is invertible, that is, if $\Delta_n\ne 0$.
\end{proof}

\begin{definition}
A linear functional $L$ defined in ${\mathcal L}$ given by (\ref{L}) is called {\em quasi-definite} if there exists a system of orthogonal Laurent polynomials with respect to $L$. $L$ is {\em positive definite} if it is quasi-definite and $L(\psi^2)>0$, $\forall \psi \in {\mathcal L}$, $\psi \neq 0$.
\end{definition}

\begin{proposition}
If $L$ is a positive definite moment functional then there exists a system of orthonormal Laurent polynomials with respect to $L$.
\end{proposition}

\begin{proof}
Suppose that $L$ is positive definite. Let $a=(a_0,\dots,a_n)$, with $a_j\in {\mathcal M}_{1,j+1}$, be an eigenvector of the matrix $\mathrm{M}_n$ corresponding to eigenvalue $\lambda$. Then, on the one hand, $a^T\,\mathrm{M}_n\,a = \lambda \|a\|^2$. On the other hand, $a^T\,\mathrm{M}_n\,a  = L(\psi^2) > 0$, where $\psi = \sum_{j=0}^{n} a_j^T \phi_j$. It follows that $\lambda > 0$. Since all the eigenvalues are positive, $\Delta_n= \det(\mathrm{M}_n)> 0$.

As a consequence, there exists a system $\{\varphi_n\}_n$ of orthogonal Laurent polynomials with respect to $L$ with ${\mathcal H}_n = L(\varphi_n\,\varphi_n^T)$. For any nonzero vector $v$, $\psi = v\varphi_n$ is a nonzero element of ${\mathcal L}_n$. Then, $v \, {\mathcal H}_n\, v^T = L (\psi^2) > 0$, so ${\mathcal H}_n$ is a positive definite matrix. If we define $\tilde{\varphi}_n = ({\mathcal H}_n^{1/2})^{-1}\varphi_n$, then $L(\tilde{\varphi}_n \, \tilde{\varphi}_n^T) = ({\mathcal H}_n^{1/2})^{-1} L(\varphi_n\, \varphi_n^T)({\mathcal H}_n^{1/2})^{-1} = {\mathcal I}_{n+1}$. This proves that $\{\tilde{\varphi}_n\}_n$ is a system of orthonormal Laurent polynomials with respect to $L$.
\end{proof}

From now on, we will deal with a positive Borel measure $\mu(x,y)$ on $\RR^2$ such that $\{ x=0 \}\cup \{ y=0 \} \not\in \textrm{supp}(\mu)=:D$. We consider the induced inner product
\begin{equation}\label{inner}
\langle f,g \rangle_{\mu}=\iint_{D} f(x,y)g(x,y)d\mu(x,y), \quad f,g \in L_2^{\mu}=\left\{ h: \RR^2 \rightarrow \RR \;:\;  \iint_{D} h^2(x,y)d\mu(x,y) < \infty \right\},
\end{equation}
and we assume the existence of the moments
\begin{equation}\label{moments}
\mu_{i,j}=\langle x^{i},y^{j} \rangle_{\mu}, \quad \quad \forall \; i,j \in \ZZ.
\end{equation}

Consider the inner product
\begin{equation}\label{innervector2}
\langle f , g^T \rangle = \left( \langle f_i , g_j \rangle_{\mu} \right)_{i=1,\ldots,k; \; j=1,\ldots,m} \in {\mathcal M}_{k,m}, \quad \textrm{where} \quad f=[f_1,\ldots,f_k]^T\, , \;\; g=[g_1,\ldots,g_m]^T.
\end{equation}

From orthogonalization to ${\mathcal L}_n = \Span \{ \phi_0,\ldots,\phi_n \}$ with respect to the inner product (\ref{innervector2}), for all $n \geq 0$, we can obtain an equivalent system ${\mathcal L}_n = \Span \{ \varphi_0,\ldots,\varphi_n \}$ verifying $\varphi_0 \in {\mathcal L}_0$, $\varphi_l \in {\mathcal L}_l \backslash {\mathcal L}_{l-1}$, $\varphi_l \perp {\mathcal L}_{l-1}$, for all $l=1,\ldots,n$, and $\langle \varphi_k , \varphi_k^T \rangle = {\mathcal I}_{k+1}$, for all $k=0,\ldots,n$. If this procedure is repeated for all $n\geq 0$, we get $\{ \varphi_k \}_{k\geq 0}$, a family of orthonormal Laurent polynomials of two real variables with respect to the measure $\mu$.

\begin{remark}\label{unicity}
Observe that $\varphi_n$ is uniquely determined up to left multiplication by orthogonal matrices. Indeed, if $Q_{n+1}\in {\mathcal M}_{n+1}$ is an orthogonal matrix and $\tilde{\varphi}_n=Q_{n+1}\varphi_n$ then $\tilde{\varphi}_n \perp {\mathcal L}_{n-1}$ and $$\langle \tilde{\varphi}_n, \tilde{\varphi}_n^T \rangle = \langle Q_{n+1}\varphi_n , \varphi_n^T Q_{n+1}^T \rangle = Q_{n+1} \langle \varphi_n , \varphi_n^T \rangle Q_{n+1}^T = {\mathcal I}_{n+1}.$$
\end{remark}

From (\ref{mult1})-(\ref{mult2}) we get
$$\left( x + \frac{1}{x} \right)\phi_n = B_{n+2,1}^{(n)}\phi_{n+2} + B_{n+1,1}^{(n)}\phi_{n+1} + B_{n-1,1}^{(n)}\phi_{n-1} + B_{n-2,1}^{(n)} \phi_{n-2},$$
where by introducing $z_s=\left( 0 \;\; \cdots \;\; 0 \;\;1 \;\; 0 \right) \in {\mathcal M}_{1,s}$ for all $s \geq 3$, it follows for all $n \geq 2$ that
\begin{equation}\label{b1}
\begin{array}{lcl}
B_{n-2,1}^{(n)}=\left[\begin{array}{c}{\mathcal I}_{n-1} \\ \rule{1cm}{0.1mm} \\ {\mathcal O}_{2,n-1}\end{array}\right] \in {\mathcal M}_{n+1,n-1}, &\quad &B_{n-1,1}^{(n)}=\left[{\mathcal O}_{n+1,n-1} | z_{n+1}^{T} \right] \in {\mathcal M}_{n+1,n}, \\ \\
B_{n+1,1}^{(n)}=\left[\begin{array}{c}{\mathcal O}_{n,n+2} \\ \rule{1cm}{0.1mm} \\ z_{n+2} \end{array}\right]  \in {\mathcal M}_{n+1,n+2}, &\quad &B_{n+2,1}^{(n)}=\left[{\mathcal I}_{n+1} | {\mathcal O}_{n+1,2}\right] \in {\mathcal M}_{n+1,n+3}.
\end{array}
\end{equation}
These formulas are also valid to define $B_{i,1}^{(0)}$ and $B_{j,1}^{(1)}$ for $i=1,2$ and $j=0,2,3$ if we interpret ${\mathcal O}_{0,2}={\mathcal O}_{2,0}=\emptyset$ and $z_2=(1 \;0)$. Here, the second subindex in the $B_{s,1}^{(n)}\in {\mathcal M}_{n+1,s+1}$ matrices with $s \in \{ n-2,n-1,n+1,n+2\}$ is used to separate the case of multiplication by $\left(y+\frac{1}{y}\right)$, see further. So, it is clear from \eqref{varphisexpand} that
\begin{equation}\label{lt1}
\left( x + \frac{1}{x} \right)\varphi_n(x,y) = A_n^{(n)}\left[ B_{n+2,1}^{(n)}\phi_{n+2} + B_{n+1,1}^{(n)}\phi_{n+1} + B_{n-1,1}^{(n)}\phi_{n-1} + B_{n-2,1}^{(n)} \phi_{n-2} \right] \;+\;\textrm{l. t.},
\end{equation}
where by ``l. t. (lower terms)" we understand linear combinations of $\left\{ \phi_0,\ldots,\phi_{n+1} \right\}$.

The main reason why multiplication by $x + \frac{1}{x}$ should be considered is the fact that $B_{n+2,1}^{(n)}$ is full rank. This also holds for multiplication by $y+\frac{1}{y}$, as we will see further, but it is easy to check from (\ref{mult1})-(\ref{mult4}) that this property is not satisfied when considering multiplication by $x+\frac{1}{y}$ or $y+\frac{1}{x}$. So, for certain constant matrices $C_k^{(n+2)}\in {\mathcal M}_{n+3,k+1}$, it holds that $\phi_{n+2}(x,y)=\sum_{k=0}^{n+2} C_k^{(n+2)}\varphi_k(x,y)$ with $C_{n+2}^{(n+2)}=\left(A_{n+2}^{(n+2)} \right)^{-1}$ and then
$$\begin{array}{ccl} \left( x + \frac{1}{x} \right)\varphi_n(x,y) &=& A_n^{(n)} B_{n+2,1}^{(n)} \phi_{n+2}(x,y) \;+\;\textrm{lower terms} \\ \\
&=& A_n^{(n)} B_{n+2,1}^{(n)}\left( \sum_{k=0}^{n+2} C_k^{(n+2)}\varphi_k(x,y)\right) \;+\;\textrm{lower terms} \\ \\
&=& A_n^{(n)} B_{n+2,1}^{(n)} C_{n+2}^{(n+2)}\varphi_{n+2}(x,y)\;+\;\textrm{lower terms}.
      \end{array}$$
If we define
$$D_{n+2,1}^{(n)}:=A_n^{(n)} B_{n+2,1}^{(n)} C_{n+2}^{(n+2)}=A_n^{(n)} B_{n+2,1}^{(n)} \left( A_{n+2}^{(n+2)} \right)^{-1} \in {\mathcal M}_{n+1,n+3},$$
and it has (full) rank $n+1$.

In short, we have proved that
$$\left( x + \frac{1}{x} \right)\varphi_n(x,y) = \sum_{k=0}^{n+2} D_{k,1}^{(n)}\varphi_{k}(x,y) \quad \textrm{where} \quad D_{k,1}^{(n)}\in {\mathcal M}_{n+1,k+1}$$
and $D_{n+2,1}^{(n)}$ being of (full) rank $n+1$.
Now, using the orthogonality conditions and the property $\left( x + \frac{1}{x} \right)\varphi_k(x,y)\in {\mathcal L}_{k+2}$, it follows that
$$D_{k,1}^{(n)} = \langle \left( x + \frac{1}{x} \right)\varphi_n(x,y) , \varphi_k^T(x,y)  \rangle = \langle \varphi_n(x,y) , \left( x + \frac{1}{x} \right)\varphi_k^T(x,y) \rangle = 0 \quad \textrm{if} \quad k < n-2.$$
This implies the following five-term recurrence relation that holds for $n \geq 2$:
\begin{equation}\label{rec}
\begin{aligned}
\left( x + \frac{1}{x} \right)\varphi_n(x,y) = &D_{n+2,1}^{(n)}\varphi_{n+2}(x,y)+D_{n+1,1}^{(n)}\varphi_{n+1}(x,y)+D_{n,1}^{(n)}\varphi_{n}(x,y)
\\
& + D_{n-1,1}^{(n)}\varphi_{n-1}(x,y)+D_{n-2,1}^{(n)}\varphi_{n-2}(x,y),
\end{aligned}
\end{equation}
with
\begin{equation}\label{dsn}
D_{s,1}^{(n)} = \langle \left( x + \frac{1}{x} \right)\varphi_n(x,y),\varphi_s^T(x,y) \rangle\in {\mathcal M}_{n+1,s+1}, \quad \quad s \in \{n-2,\ldots,n+2 \}
\end{equation} and leading coefficient matrix $D_{n+2,1}^{(n)}$ of (full) rank $n+1$. Moreover, for $s \in \left\{ n-2,\ldots n+2 \right\}$
$$
\begin{aligned}
D_{s,1}^{(n)} = &\langle \left( x + \frac{1}{x} \right)\varphi_n(x,y),\varphi_{s}^T(x,y) \rangle = \langle \varphi_n(x,y),\left( x + \frac{1}{x} \right)\varphi_{s}^T(x,y) \rangle
\\
=& \langle \varphi_n, \left( D_{s+2,1}^{(s)}\varphi_{s+2} + D_{s+1,1}^{(s)}\varphi_{s+1} + D_{s,1}^{(s)}\varphi_{s}+ D_{s-1,1}^{(s)}\varphi_{s-1} + D_{s-2,1}^{(s)}\varphi_{s-2} \right)^T \rangle = \left(D_{n,1}^{(s)}\right)^T
\end{aligned}
$$
so we have proved that
\begin{equation}\label{coeftailed}
 D_{s,1}^{(n)} = \left(D_{n,1}^{(s)}\right)^T, \quad s \in \left\{ n-2,\ldots n+2 \right\}.
\end{equation}
This implies, in particular, that $D_{n,1}^{(n)}$ is symmetric and that the tailed coefficient matrix $D_{n-2,1}^{(n)}\in {\mathcal M}_{n+1,n-1}$ in (\ref{rec}) is also of full rank, equal to $n-1$.

Concerning the initial conditions, we may observe that the recurrence (\ref{rec}) is also valid for $n=0,1$ by setting $\varphi_{-1} \equiv \varphi_{-2} \equiv 0$. Indeed, recall first that
$$\phi_0(x,y) \equiv 1, \quad \phi_1(x,y) = \left( \begin{array}{c} x \\ y  \end{array}\right), \quad \phi_2(x,y)=\left( \begin{array}{c} 1/x \\ xy \\ 1/y  \end{array}\right), \quad \phi_3(x,y)=\left( \begin{array}{c} x^2 \\ y/x \\ x/y \\ y^2\end{array}\right).$$
For $n=0$, we can find matrices $D_{i,1}^{(0)} \in {\mathcal M}_{1,i+1}$, $i=0,1,2$ such that (\ref{rec}) holds. From (\ref{varphisexpand}) we can write
$$\left( x + \frac{1}{x} \right)A_0^{(0)}=\left[ D_{2,1}^{(0)}A_2^{(2)} \right]\phi_2 + \left[ D_{1,1}^{(0)}A_1^{(1)}+D_{2,1}^{(0)}A_1^{(2)} \right]\phi_1+\left[ D_{0,1}^{(0)}A_0^{(0)}+D_{1,1}^{(0)}A_0^{(1)}+D_{2,1}^{(0)}A_0^{(2)} \right]\phi_0$$
where $A_i^{(i)}$ are regular, for $i=0,1,2$. Hence, since
$$\left( x + \frac{1}{x} \right)A_0^{(0)}=B_{2,1}^{(0)}A_0^{(0)}\phi_2 + B_{1,1}^{(0)}A_0^{(0)}\phi_1,$$
it follows that
\begin{equation}\label{initialprocedure0}\begin{array}{l} D_{2,1}^{(0)}=B_{2,1}^{(0)}A_0^{(0)}\left(A_2^{(2)}\right)^{-1}
\\[6pt]
D_{1,1}^{(0)}=\left( B_{1,1}^{(0)} A_0^{(0)} - D_{2,1}^{(0)}A_1^{(2)} \right) \left(A_1^{(1)}\right)^{-1}
\\[6pt]
D_{0,1}^{(0)}=-\left( D_{1,1}^{(0)}A_0^{(1)} + D_{2,1}^{(0)}A_0^{(2)}\right) \left(A_0^{(0)}\right)^{-1}. \end{array}\end{equation}
Similarly when $n=1$, we can find matrices $D_{i,1}^{(1)} \in {\mathcal M}_{2,i+1}$, $i=0,1,2,3$ such that (\ref{rec}) holds. By one hand, we can write from (\ref{varphisexpand})
$$\begin{aligned}
\left( x + \frac{1}{x} \right)\left[ A_0^{(1)}\phi_0 + A_1^{(1)}\phi_1  \right]= & \left[ D_{3,1}^{(1)}A_3^{(3)} \right]\phi_3 + \left[ D_{3,1}^{(1)}A_2^{(3)} + D_{2,1}^{(1)}A_2^{(2)} \right]\phi_2
\\[6pt]
&  + \left[ D_{3,1}^{(1)}A_1^{(3)}+D_{2,1}^{(1)}A_1^{(2)}+D_{1,1}^{(1)}A_1^{(1)} \right]\phi_1
\\[6pt]
&
+\left[ D_{3,1}^{(1)}A_0^{(3)}+D_{2,1}^{(1)}A_0^{(2)}+D_{1,1}^{(1)}A_0^{(1)}+D_{0,1}^{(1)}A_0^{(0)} \right]\phi_0
\end{aligned}$$
with $A_i^{(i)}$ regular matrices for $i=0,\ldots,3$. By other hand,
$$\left( x + \frac{1}{x} \right)\left[ A_0^{(1)}\phi_0 + A_1^{(1)}\phi_1  \right]= A_1^{(1)}B_{3,1}^{(1)}\phi_3 + \left[ A_0^{(1)} B_{2,1}^{(0)} + A_1^{(1)} B_{2,1}^{(1)} \right]\phi_2 + A_0^{(1)} B_{1,1}^{(0)}\phi_1 + A_1^{(1)}B_{0,1}^{(1)}\phi_0.$$
Hence,
\begin{equation}\label{initialprocedure1}\begin{array}{l} D_{3,1}^{(1)}=A_1^{(1)}B_{3,1}^{(1)} \cdot\left(A_3^{(3)}\right)^{-1}
\\[5pt]
D_{2,1}^{(1)}=\left( A_0^{(1)} B_{2,1}^{(0)} + A_1^{(1)} B_{2,1}^{(1)} - D_{3,1}^{(1)}A_2^{(3)} \right) \cdot\left(A_2^{(2)}\right)^{-1}
\\[5pt]
D_{1,1}^{(1)}=\left( A_0^{(1)}B_{1,1}^{(0)} - D_{3,1}^{(1)}A_1^{(3)} - D_{2,1}^{(1)}A_1^{(2)}\right)\cdot\left(A_1^{(1)}\right)^{-1}
\\[5pt]
D_{0,1}^{(1)}=\left( A_1^{(1)}B_{0,1}^{(1)} - D_{3,1}^{(1)}A_0^{(3)} - D_{2,1}^{(1)}A_0^{(2)} - D_{1,1}^{(1)}A_0^{(1)}\right)\cdot\left(A_0^{(0)}\right)^{-1}.  \end{array}\end{equation}

As a consequence, we can give from (\ref{coeftailed}) a matrix representation with respect to $\{ \varphi_k \}_{k=0}^{\infty}$ of the {\em multiplication plus inverse multiplication} operator:
\begin{equation}\label{representation1x}
\left( x + \frac{1}{x} \right) \cdot \left(\varphi_0 \;\; \varphi_1 \;\; \varphi_2 \;\; \varphi_3 \;\; \varphi_4 \;\; \varphi_5 \;\; \cdots \right)^{T}={\mathcal F}_1 \cdot \left( \varphi_0 \;\; \varphi_1 \;\; \varphi_2 \;\; \varphi_3 \;\; \varphi_4 \;\; \varphi_5 \;\; \cdots \right)^{T}
\end{equation}
where ${\mathcal F}_1$ is the block symmetric matrix given by
\begin{equation}\label{representation2x}{\mathcal F}_1=\left(\begin{array}{ccccccccc} D_{0,1}^{(0)} & D_{1,1}^{(0)} & D_{2,1}^{(0)} & {\mathcal O}_{1,4} & {\mathcal O}_{1,5} & {\mathcal O}_{1,6} & {\mathcal O}_{1,7} & {\mathcal O}_{1,8} & \cdots \\ \left( D_{1,1}^{(0)}\right)^T & D_{1,1}^{(1)} & D_{2,1}^{(1)} & D_{3,1}^{(1)} & {\mathcal O}_{2,5} & {\mathcal O}_{2,6} & {\mathcal O}_{2,7} & {\mathcal O}_{2,8} & \cdots \\ \left(D_{2,1}^{(0)}\right)^T & \left( D_{2,1}^{(1)}\right)^T & D_{2,1}^{(2)} & D_{3,1}^{(2)} & D_{4,1}^{(2)} & {\mathcal O}_{3,6} & {\mathcal O}_{3,7} & {\mathcal O}_{3,8} & \cdots \\ {\mathcal O}_{4,1} & \left(D_{3,1}^{(1)}\right)^T & \left( D_{3,1}^{(2)}\right)^T & D_{3,1}^{(3)} & D_{4,1}^{(3)} & D_{5,1}^{(3)} & {\mathcal O}_{4,7} & {\mathcal O}_{4,8} & \cdots \\ {\mathcal O}_{5,1} & {\mathcal O}_{5,2} & \left(D_{4,1}^{(2)}\right)^T & \left( D_{4,1}^{(3)}\right)^T & D_{4,1}^{(4)} & D_{5,1}^{(4)} & D_{6,1}^{(4)} & {\mathcal O}_{5,8} & \cdots \\ \vdots & \vdots & \vdots & \vdots & \vdots & \vdots & \vdots & \vdots & \ddots  \end{array}\right).
\end{equation}

A very similar analysis can be done from (\ref{mult3})-(\ref{mult4}) when considering multiplication by $y+\frac{1}{y}$, we omit most of the details.
It is important to point out that the corresponding $B_{s,2}^{(n)}$ matrices in the relation
$$\left( y + \frac{1}{y} \right)\phi_n = B_{n+2,2}^{(n)}\phi_{n+2} + B_{n+1,2}^{(n)}\phi_{n+1} + B_{n-1,2}^{(n)}\phi_{n-1} + B_{n-2,2}^{(n)} \phi_{n-2}$$
are given in this case (compare with (\ref{b1})) for all $n \geq 2$ by
$$\begin{array}{lcl}
B_{n-2,2}^{(n)}=\left[\begin{array}{c}{\mathcal O}_{2,n-1} \\ \rule{1cm}{0.1mm} \\ {\mathcal I}_{n-1}\end{array}\right] \in {\mathcal M}_{n+1,n-1}, &\quad &B_{n-1,2}^{(n)}=\left[ \tilde{z}_{n+1}^{T} | {\mathcal O}_{n+1,n-1} \right] \in {\mathcal M}_{n+1,n}, \\ \\
B_{n+1,2}^{(n)}=\left[\begin{array}{c} \tilde{z}_{n+2} \\ \rule{1cm}{0.1mm} \\ {\mathcal O}_{n,n+2}  \end{array}\right]  \in {\mathcal M}_{n+1,n+2}, &\quad &B_{n+2,2}^{(n)}=\left[ {\mathcal O}_{n+1,2} | {\mathcal I}_{n+1} \right] \in {\mathcal M}_{n+1,n+3},
\end{array}$$
where $\tilde{z}_s=\left( 0 \;\; 1 \;\; 0 \;\; \cdots \;\; 0 \right) \in {\mathcal M}_{1,s}$ for all $s \geq 3$ and $\tilde{z}_2=(0 \;1)$. These formulas are again valid for $n=0$, $s\in \{1,2\}$ and $n=1$, $s\in \{ 0,2,3 \}$. Thus,
$$D_{n+2,2}^{(n)}:=\left[ {\mathcal O}_{n+1,2} | A_n^{(n)} \right] \cdot C_{n+2}^{(n+2)}=A_n^{(n)}\cdot B_{n+2,2}^{(n)}\cdot \left( A_{n+2}^{(n+2)} \right)^{-1} \in {\mathcal M}_{n+1,n+3},$$
also has (full) rank $n+1$. The analog of (\ref{lt1}) is
\begin{equation}\label{lt2}
\left( y + \frac{1}{y} \right)\varphi_n(x,y) = A_n^{(n)}\left[ B_{n+2,2}^{(n)}\phi_{n+2} + B_{n+1,2}^{(n)}\phi_{n+1} + B_{n-1,2}^{(n)}\phi_{n-1} + B_{n-2,2}^{(n)} \phi_{n-2} \right] \;+\;\textrm{lower terms}.
\end{equation}

The corresponding five-term recurrence is now, for $n \geq 2$:
\begin{equation}\label{rec2}
\begin{aligned}
\left( y + \frac{1}{y} \right)\varphi_n(x,y) = &  D_{n+2,2}^{(n)}\varphi_{n+2}(x,y)+D_{n+1,2}^{(n)}\varphi_{n+1}(x,y)+D_{n,2}^{(n)}\varphi_{n}(x,y)
\\
& + D_{n-1,2}^{(n)}\varphi_{n-1}(x,y)+D_{n-2,2}^{(n)}\varphi_{n-2}(x,y),
\end{aligned}
\end{equation}
with
\begin{equation}\label{dsn2}
D_{s,2}^{(n)} = \langle \left( y + \frac{1}{y} \right)\varphi_n(x,y),\varphi_s^T(x,y) \rangle\in {\mathcal M}_{n+1,s+1}, \quad \quad s \in \{n-2,\ldots,n+2 \}.
\end{equation}
Equation (\ref{rec2}) is again valid for $n=0,1$. The analog of (\ref{initialprocedure0}) is
\begin{equation}\label{initialprocedure00}\begin{array}{l} D_{2,2}^{(0)}=B_{2,2}^{(0)}A_0^{(0)}\left(A_2^{(2)}\right)^{-1}
\\[5pt]
D_{1,2}^{(0)}=\left( B_{1,2}^{(0)} A_0^{(0)} - D_{2,2}^{(0)}A_1^{(2)} \right) \left(A_1^{(1)}\right)^{-1}
\\[5pt]
D_{0,2}^{(0)}=-\left( D_{1,2}^{(0)}A_0^{(1)} + D_{2,2}^{(0)}A_0^{(2)}\right) \left(A_0^{(0)}\right)^{-1}. \end{array}\end{equation}
As in (\ref{coeftailed}) it holds
\begin{equation}\label{symmetryD2}
D_{s,2}^{(n)}=\left(D_{n,2}^{(s)}\right)^T \quad \textrm{for} \quad s \in \left\{n-2,\ldots,n+2 \right\}
\end{equation} and hence, $D_{n,2}^{(n)}$ is symmetric and the tailing coefficient matrix $D_{n-2,2}^{(n)}\in {\mathcal M}_{n+1,n-1}$ in (\ref{rec2}) is also of full rank, equal to $n-1$.

The matrix representation with respect to $\{ \varphi_k \}_{k=0}^{\infty}$ of the {\em multiplication plus inverse multiplication} operator in the variable $y$ is
\begin{equation}\label{representation1y}
\left( y + \frac{1}{y} \right) \cdot \left(\varphi_0 \;\; \varphi_1 \;\; \varphi_2 \;\; \varphi_3 \;\; \varphi_4 \;\; \varphi_5 \;\; \cdots \right)^{T}={\mathcal F}_2 \cdot \left( \varphi_0 \;\; \varphi_1 \;\; \varphi_2 \;\; \varphi_3 \;\; \varphi_4 \;\; \varphi_5 \;\; \cdots \right)^{T}
\end{equation}
where ${\mathcal F}_2$ is a block symmetric matrix like ${\mathcal F}_1$ in (\ref{representation2x}) but replacing the second subindexes $1$ in the $D_{s,1}^{(n)}$ matrices by $2$.

The results of this section can be summarized in the following

\begin{theorem}
Let $\{ \varphi_k \}_{k \geq 0}$ be a family of orthonormal Laurent polynomials with respect to the measure $\mu$.
Then, for all $n \geq 0$ there exist constant matrices $D_{s,i}^{(n)} \in {\mathcal M}_{n+1,s+1}$, $s \in \{n-2,\ldots,n+2\}$, $i\in \{ 1,2 \}$ given by (\ref{dsn}) and (\ref{dsn2}) when $n\geq 2$, $s \in \{0,1,2\}$ if $n=0$ and $s \in \{0,1,2,3\}$ if $n=1$, such that the following five-term relations hold:
\begin{equation}\label{5TR}
\begin{aligned}
\left( x + \frac{1}{x} \right)\varphi_n(x,y) =& D_{n+2,1}^{(n)}\varphi_{n+2}(x,y)+D_{n+1,1}^{(n)}\varphi_{n+1}(x,y)  +D_{n,1}^{(n)}\varphi_{n}(x,y)
\\
&+ D_{n-1,1}^{(n)}\varphi_{n-1}(x,y)+D_{n-2,1}^{(n)}\varphi_{n-2}(x,y),
\\
\left( y + \frac{1}{y} \right)\varphi_n(x,y) =&  D_{n+2,2}^{(n)}\varphi_{n+2}(x,y)+D_{n+1,2}^{(n)}\varphi_{n+1}(x,y) +D_{n,2}^{(n)}\varphi_{n}(x,y)
\\
&+ D_{n-1,2}^{(n)}\varphi_{n-1}(x,y)+D_{n-2,2}^{(n)}\varphi_{n-2}(x,y),
\end{aligned}
\end{equation}
with $\varphi_{-1} \equiv \varphi_{-2} \equiv 0$.  Moreover, for $i\in \{ 1,2 \}$, $D_{s,i}^{(n)}=\left( D_{n,i}^{(s)} \right)^T$, for all $n \geq 2$, $s \in \left\{ n-2,n-1,n \right\}$,  $D_{0,i}^{(1)}=\left( D_{1,i}^{(0)} \right)^T$ and the matrices $D_{n+2,i}^{(n)}$ and  $D_{n-2,i}^{(n)}$ are full rank.
\end{theorem}

We conclude this section with the following elementary result that will be used in the next section.
\begin{lemma}\label{lemaD}
For all $n \geq 1$, the matrices
$$D_{n+2}^{(n)}:=\left[ \begin{array}{c} D_{n+2,1}^{(n)} \\ \rule{1cm}{0.1mm} \\ D_{n+2,2}^{(n)} \end{array}\right] \in {\mathcal M}_{2(n+1),n+3} \quad \quad \textrm{and} \quad \quad B_{n+2}^{(n)}:=\left[ \begin{array}{c} B_{n+2,1}^{(n)} \\ \rule{1cm}{0.1mm} \\ B_{n+2,2}^{(n)} \end{array}\right] \in {\mathcal M}_{2(n+1),n+3}$$
have full rank, equal to $n+3$. This full rank properties are also valid for $n=0$, but in this case they are equal to $2$, instead of $3$.
\end{lemma}

\begin{proof}
The result trivially follows for the matrix $B_{n+2}^{(n)}$ since we have already seen that
$$B_{n+2,1}^{(n)}=[{\mathcal I}_{n+1} | {\mathcal O}_{n+1,2}]\in {\mathcal M}_{n+1,n+3}, \quad \quad \quad B_{n+2,2}^{(n)}=[{\mathcal O}_{n+1,2} | {\mathcal I}_{n+1} ]\in {\mathcal M}_{n+1,n+3}.$$
We can write
$$D_{n+2}^{(n)}=\left( \begin{array}{cc}
A_n^{(n)}   &  {\mathcal O}_{n+1,2}
\\[5pt] \hline
{\mathcal O}_{n+1,2} &  A_n^{(n)}
\end{array}\right)\cdot \left( A_{n+2}^{(n+2)} \right)^{-1},$$  so the result follows directly by using Sylvester inequality (see e.g. \cite[p. 13]{Horn}). Finally, the result for $n=0$ holds since
$$D_2^{(0)}=B_2^{(0)}\cdot \left( A_2^{(2)} \right)^{-1}.$$
\end{proof}

\section{Favard's theorem and Christoffel-Darboux formula}\label{SecFavCD}

From the results of Section \ref{SecOLP2VRR} we have all the necessary technical modifications to adapt the proofs of Favard's theorem and Christoffel-Darboux formula presented in \cite[Section 3.3]{Xu} for the ordinary polynomials in several variables to the Laurent case.

Our first observation is that since for all $n\geq 1$, the matrix $D_{n+2}^{(n)} \in {\mathcal M}_{2(n+1),n+3}$ defined in Lemma \ref{lemaD} is of full rank $n+3$, it has a left inverse
$$\left( \overline{D}_{n+2}^{(n)}\right)^T=\left( \left(\overline{D}_{n+2,1}^{(n)}\right)^T | \left(\overline{D}_{n+2,2}^{(n)}\right)^T \right) \in {\mathcal M}_{n+3,2(n+1)}, \quad \quad \left(\overline{D}_{n+2,i}^{(n)}\right)^T\in {\mathcal M}_{n+3,n+1}, \quad i=1,2,$$ that is not unique. This means
$$\left( \overline{D}_{n+2}^{(n)}\right)^T \cdot D_{n+2}^{(n)} = \left( \overline{D}_{n+2,1}^{(n)}\right)^T \cdot  D_{n+2,1}^{(n)} + \left( \overline{D}_{n+2,2}^{(n)}\right)^T \cdot  D_{n+2,2}^{(n)} = {\mathcal I}_{n+3}.$$
We need also the following auxiliary result.
\begin{proposition}\label{favardaux}
Let $\left( \overline{D}_{n+2}^{(n)}\right)^T$ be a left inverse of $\overline{D}_{n+2}^{(n)}$. Then, there exists constant matrices $E_n^{i} \in {\mathcal M}_{n+3,n+3-i}$, $i=1,2,3,4$ such that
$$\varphi_{n+2}=\left[ \left(x+\frac{1}{x} \right)\left(\overline{D}_{n+2,1}^{(n)}\right)^T + \left(y+\frac{1}{y} \right)\left(\overline{D}_{n+2,2}^{(n)}\right)^T \right]\varphi_n + E_n^{1}\varphi_{n+1}+ E_n^{2}\varphi_{n}+ E_n^{3}\varphi_{n-1}+ E_n^{4}\varphi_{n-2}.$$
\end{proposition}

\begin{proof}
If we add (\ref{rec}) multiplied on the left by $\left(\overline{D}_{n+2,1}^{(n)}\right)^T$ and (\ref{rec2}) multiplied on the left by $\left(\overline{D}_{n+2,2}^{(n)}\right)^T$, we get
$$\begin{array}{l} \left[ \left(x+\frac{1}{x} \right)\left(\overline{D}_{n+2,1}^{(n)}\right)^T + \left(y+\frac{1}{y} \right)\left(\overline{D}_{n+2,2}^{(n)}\right)^T \right]\varphi_n \; = \; \left[ \left(\overline{D}_{n+2,1}^{(n)}\right)^T D_{n+2,1}^{(n)} + \left(\overline{D}_{n+2,2}^{(n)}\right)^T D_{n+2,2}^{(n)}\right]\varphi_{n+2} \\ \quad + \left[ \left(\overline{D}_{n+2,1}^{(n)}\right)^T D_{n+1,1}^{(n)} + \left(\overline{D}_{n+2,2}^{(n)}\right)^T D_{n+1,2}^{(n)}\right]\varphi_{n+1} + \left[ \left(\overline{D}_{n+2,1}^{(n)}\right)^T D_{n,1}^{(n)} + \left(\overline{D}_{n+2,2}^{(n)}\right)^T D_{n,2}^{(n)}\right]\varphi_{n} \\ \quad + \left[ \left(\overline{D}_{n+2,1}^{(n)}\right)^T D_{n-1,1}^{(n)} + \left(\overline{D}_{n+2,2}^{(n)}\right)^T D_{n-1,2}^{(n)}\right]\varphi_{n-1} + \left[ \left(\overline{D}_{n+2,1}^{(n)}\right)^T D_{n-2,1}^{(n)} + \left(\overline{D}_{n+2,2}^{(n)}\right)^T D_{n-2,2}^{(n)}\right]\varphi_{n-2}.\end{array}$$
So, the result follows by considering
$$E_n^{i}=-\left[  \left(\overline{D}_{n+2,1}^{(n)}\right)^T D_{n+2-i,1}^{(n)} + \left(\overline{D}_{n+2,2}^{(n)}\right)^T D_{n+2-i,2}^{(n)} \right] \in {\mathcal M}_{n+3,n+3-i}, \quad \quad i=1,2,3,4.$$
\end{proof}

Now we are in position to prove a Favard-type theorem by following the ideas presented in \cite[Section 3.3]{Xu}. We concentrate in the positive-definite case.
\begin{theorem}[Favard]\label{favard1}
Let ${\mathcal L}$ and ${\mathcal L}_n$ be given by (\ref{matrixL})-(\ref{Ln}), $\{ \varphi_n \}_{n=0}^{\infty}$ be an arbitrary sequence in ${\mathcal L}$ written in the form (\ref{varphisexpand}) with $\phi_k$ defined in (\ref{phik}) for all $k \geq 0$, $\varphi_n \in {\mathcal L}_n \backslash {\mathcal L}_{n-1}$, for all $n \geq 1$ where $\varphi_0 \equiv 1$ and set $\varphi_{-2} \equiv \varphi_{-1} \equiv 0$.

Suppose that for all $n \geq 0$, there exist matrices $D_{k,i}^{(n)}\in {\mathcal M}_{n+1,k+1}$, $i=1,2$, $k \in \{ n-2,\ldots,n+2\}$ when $n\geq 2$, $k \in \{ 0,1,2 \}$ when $n=0$ and $k \in \{ 0,1,2,3 \}$ when $n=1$, such that
\begin{enumerate}
\item the $L$-polynomials $\varphi_n$ satisfy the recurrences \eqref{5TR} with $D_{s,i}^{(n)}=\left( D_{n,i}^{(s)} \right)^T$, for all $n \geq 2$, $s \in \{ n-2,n-1,n \}$, $D_{0,i}^{(1)}=\left( D_{1,i}^{(0)} \right)^T$ and $i\in \{1,2\}$.
\item the matrices in the relation satisfy the rank conditions:
    $$
    \begin{aligned}
    &\rank D_{n+2,i}^{(n)}=\rank D_{n,i}^{(n+2)}=n+1, \qquad n \geq 0, \quad i=1,2, \\
    &\rank D_{n+2}^{(n)}=n+3, \qquad n \geq 1 \\
    &\rank D_{1}^{(0)}=2,
    \end{aligned}
    $$
with $D_{n+2}^{(n)} \in {\mathcal M}_{2(n+1),n+3}$ introduced in Lemma \ref{lemaD} and $D_{1}^{(0)} \in {\mathcal M}_2$ given by
\begin{equation}\label{d10}
D_1^{(0)}:=\left[ \begin{array}{c} D_{1,1}^{(0)} \\ \rule{1cm}{0.1mm} \\ D_{1,2}^{(0)} \end{array}\right]\in {\mathcal M}_2 \, .
\end{equation}
\end{enumerate}
Then, there exist a linear functional $L$ which defines a positive-definite functional on ${\mathcal L}$ and which makes $\{ \varphi_n \}_{n=0}^{\infty}$ an orthonormal basis in ${\mathcal L}$.
\end{theorem}

\begin{proof}
We first prove that $\{ \varphi_n \}_{n=0}^{\infty}$ forms a basis of ${\mathcal L}$. Using the expression (\ref{varphisexpand}), it suffices to prove that the leading coefficient $A_n^{(n)}$ is regular, for all $n \geq 0$. For $n \geq 2$ we see that comparing the coefficient matrices of $\phi_{n+2}$ in (\ref{lt1}), (\ref{lt2}) and (\ref{rec}), (\ref{rec2}) we get
$$\textrm{diag}\left( A_n^{(n)} \;,\; A_n^{(n)}  \right)\cdot B_{n+2}^{(n)}=D_{n+2}^{(n)}\cdot A_{n+2}^{(n+2)},$$ where $B_{n+2}^{(n)}, D_{n+2}^{(n)}$ are the matrices of rank $n+3$ that have been introduced in Lemma \ref{lemaD}. To prove that $\rank A_n^{(n)}=n+1$ we proceed by induction by showing that from the two initial conditions $n=0,1$ (that holds by hypothesis) we get that if $\rank A_n^{(n)}=n+1$, then $\rank A_{n+2}^{(n+2)}=n+3$. Indeed, if $A_n^{(n)}$ is invertible, then $\textrm{diag}\left( A_n^{(n)} \;,\; A_n^{(n)}  \right)$ is also invertible, $$\rank\left(\textrm{diag}\left( A_n^{(n)} \;,\; A_n^{(n)}\right)\cdot B_{n+2}^{(n)}\right)=\rank B_{n+2}^{(n)}=n+3,$$
and hence, $\rank \left( D_{n+2}^{(n)} \cdot A_{n+2}^{(n+2)} \right)=n+3$. By using Sylvester inequality we get
$$\rank A_{n+2}^{(n+2)} \geq \rank \left( D_{n+2}^{(n)} \cdot A_{n+2}^{(n+2)} \right) \geq \rank D_{n+2}^{(n)} + \rank A_{n+2}^{(n+2)}-(n+3) =\rank A_{n+2}^{(n+2)}.$$
So, we conclude $\rank A_{n+2}^{(n+2)}=\rank \left( D_{n+2}^{(n)} \cdot A_{n+2}^{(n+2)} \right)=n+3$ and the induction is complete.

Since $\{ \varphi_n \}_{n=0}^{\infty}$ is a basis of ${\mathcal L}$, the linear functional $L$ defined on ${\mathcal L}$ by $L(1)=1$ and $L(\varphi_n)=0$, for all $n\geq 1$ is well defined. We now use induction to prove that
\begin{equation}\label{favarddef}L(\varphi_k \varphi_j^T)=0, \quad \quad \textrm{for all} \;\; k \neq j.
\end{equation}
For $n\geq 0$ assume that (\ref{favarddef}) hold $\forall k,j$ such that $0 \leq k \leq n$ and $j>k$. The induction process is directly obtained from Proposition \ref{favardaux} since we have for all $l>n+1$ that
$$L(\varphi_{n+1}\varphi_l^T)=L\left[ \left( \overline{D}_{n+1,1}^{(n-1)} \right)^T \varphi_{n-1} \left(x+\frac{1}{x}\right)\varphi_l^T  \right]+L\left[ \left( \overline{D}_{n+1,2}^{(n-1)} \right)^T \varphi_{n-1}\left(y+\frac{1}{y}\right)\varphi_l^T  \right]=0.$$
Let us see finally that ${\mathcal H}_n:=L(\varphi_n \varphi_n^T)={\mathcal I}_{n+1}$. Notice from (\ref{rec}) that
$$
D_{n+2,1}^{(n)}{\mathcal H}_{n+2}=L\left[ \left( x + \frac{1}{x}\right)\varphi_n \varphi_{n+2}^T \right]=L\left[ \varphi_n \left( \left( x + \frac{1}{x}\right)\varphi_{n+2} \right)^T \right]={\mathcal H}_n \left( D_{n,1}^{(n+2)}\right)^T = {\mathcal H}_n D_{n+2,1}^{(n)}, \quad \forall n \geq 0.
$$
We get a similar result from (\ref{rec2}) when multiplying by $\left( y + \frac{1}{y} \right)$, and both relations can be written together as
\begin{equation}\label{favardh}
D_{n+2}^{(n)}\cdot {\mathcal H}_{n+2} = \textrm{diag}\left( {\mathcal H}_n \;,\; {\mathcal H}_n\right)\cdot D_{n+2}^{(n)}.
\end{equation}
We proceed again by induction over $n$. It is clear from construction that it holds for $n=0$ and the proof is concluded if we prove it for $n=1$. Indeed, in such case if we suppose that the property holds for all $0 \leq k \leq n+1$, it follows from (\ref{favardh}) that $D_{n+2}^{(n)}\cdot {\mathcal H}_{n+2} = D_{n+2}^{(n)}$. Since $D_{n+2}^{(n)}$ is of full rank it has a left inverse, so ${\mathcal H}_{n+2}={\mathcal I}_{n+2}$.

Taking $n=0,1$ in (\ref{rec}) we see that
$$L\left[ \left( x + \frac{1}{x} \right) \varphi_0 \varphi_1^T \right]=D_{2,1}^{(0)}L[\varphi_2 \varphi_1^T] + D_{1,1}^{(0)}L[\varphi_1 \varphi_1^T]+ D_{0,1}^{(0)}L[\varphi_0 \varphi_1^T] = D_{1,1}^{(0)}L[\varphi_1 \varphi_1^T]$$
and
$$
\begin{aligned}
L\left[ \left( x + \frac{1}{x} \right) \varphi_0 \varphi_1^T \right] &=L[\varphi_3]\left( D_{3,1}^{(1)} \right)^T + L[\varphi_2]\left( D_{2,1}^{(1)} \right)^T + L[\varphi_1]\left( D_{1,1}^{(1)} \right)^T + L[\varphi_0]\left( D_{0,1}^{(1)} \right)^T
\\
 &= \left( D_{0,1}^{(1)} \right)^T = D_{1,1}^{(0)}.
\end{aligned}
$$
The same argument can be used to prove $L\left[ \left( y + \frac{1}{y} \right) \varphi_0 \varphi_1^T \right]=D_{1,2}^{(0)}L[\varphi_1 \varphi_1^T]=D_{1,2}^{(0)}$, so we get $D_1^{(0)}L\left[\varphi_1 \varphi_1^T\right]=D_1^{(0)}$ with $D_1^{(0)}$ introduced in (\ref{d10}).
Thus, the proof follows since $D_1^{(0)}$ is regular.
\end{proof}

Let us introduce now the {\em reproducing kernel}
\begin{equation}\label{eq-ker}
{\mathcal K}_n(x_1,y_1,x_2,y_2)=\sum_{k=0}^{n} \varphi_k^T(x_1,y_1)\varphi_k(x_2,y_2).
\end{equation}
This definition is clearly independent on the election of the orthonormal family $\{ \varphi_n \}_{n\geq 0}$ (recall it is uniquely determined up to left multiplication by orthogonal matrices). The name reproducing kernel is justified as in the ordinary polynomial situation because it is easy to verify the reproducing property $\psi(x,y)=\langle \psi(u,v),{\mathcal K}_n^T(x,y,u,v)\rangle$, $\forall \psi \in {\mathcal L}_n$. The extension of the well known Christoffel-Darboux formula for the ordinary polynomial situation (see \cite[Section 3.6.1]{Xu}) is given by the following

\begin{theorem}[Christoffel-Darboux]\label{cd}
Under the above conditions it holds
$${\mathcal K}_n(x_1,y_1,x_2,y_2)=\frac{\Omega_{n,1} + \Lambda_{n,1} + \Lambda_{n-1,1}}{\left( x_1 + \frac{1}{x_1}\right) - \left( x_2 + \frac{1}{x_2}\right)}=\frac{\Omega_{n,2} + \Lambda_{n,2} + \Lambda_{n-1,2}}{\left( y_1 + \frac{1}{y_1}\right) - \left( y_2 + \frac{1}{y_2}\right)},$$
whenever $x_1 + \frac{1}{x_1} \neq x_2 + \frac{1}{x_2}$ in the first equality, $y_1 + \frac{1}{y_1} \neq y_2 + \frac{1}{y_2}$ in the second one, and where for $i=1,2$ and $k \geq 0$,
\begin{equation}\label{LambdaOmegaki}
\begin{array}{ccl}
\Lambda_{k,i} &=&\varphi_{k+2}^T(x_1,y_1)\left( D_{k+2,i}^{(k)} \right)^T \varphi_k(x_2,y_2) - \varphi_{k}^T(x_1,y_1)D_{k+2,i}^{(k)}\varphi_{k+2}(x_2,y_2), \\
\Omega_{k,i} &=&\varphi_{k+1}^T(x_1,y_1)\left( D_{k+1,i}^{(k)} \right)^T \varphi_k(x_2,y_2) - \varphi_{k}^T(x_1,y_1)D_{k+1,i}^{(k)}\varphi_{k+1}(x_2,y_2).
\end{array}
\end{equation}
\end{theorem}

\begin{proof}
From (\ref{rec}) and (\ref{coeftailed}) we can write for all $k \geq 2$ and $x_1 + \frac{1}{x_1} \neq x_2 + \frac{1}{x_2}$,
$$\left[ \left( x_1 + \frac{1}{x_1}\right) - \left( x_2 + \frac{1}{x_2}\right) \right] \varphi_k^T(x_1,y_1)\varphi_k(x_2,y_2) = \Lambda_{k,1} + \Omega_{k,1} - \Lambda_{k-2,1} - \Omega_{k-1,1}.$$
Taking $n=0,1$ in (\ref{rec}) it follows that the relation also holds for $k=0$ and $k=1$ respectively (recall $\varphi_{-1} \equiv \varphi_{-2} \equiv 0$), if we define $\Lambda_{-2,1}=\Lambda_{-1,1}=\Omega_{-1,1}=0$. So,
$$\sum_{k=0}^{n}\left[ \left( x_1 + \frac{1}{x_1}\right) - \left( x_2 + \frac{1}{x_2}\right) \right] \varphi_k^T(x_1,y_1)\varphi_k(x_2,y_2)=\Omega_{n,1} + \Lambda_{n,1} + \Lambda_{n-1,1}$$
and the first equality of the statement is deduced. The second equality follows in a similar way from (\ref{rec2}) and (\ref{symmetryD2}).
\end{proof}

\begin{corollary}[Confluent formula]\label{confluent}
Under the above conditions it holds
$${\mathcal K}_n(x,y,x,y)=\frac{x^2}{x^2-1}\left[ \tilde{\Omega}_{n,1} + \tilde{\Lambda}_{n,1} + \tilde{\Lambda}_{n-1,1} \right]=\frac{y^2}{y^2-1}\left[ \tilde{\Omega}_{n,2} + \tilde{\Lambda}_{n,2} + \tilde{\Lambda}_{n-1,2} \right]$$
whenever $x^2\neq 1$ in the first equality, $y^2\neq 1$ in the second one, and
$$\begin{array}{ccl}
\tilde{\Lambda}_{k,1} &=& {\displaystyle \varphi_{k+2}^T(x,y)\left( D_{k+2,1}^{(k)} \right)^T \frac{\partial}{\partial x} \varphi_k(x,y)  - \varphi_{k}^T(x,y)D_{k+2,1}^{(k)}\frac{\partial}{\partial x} \varphi_{k+2}(x,y),}
\\[7pt]
\tilde{\Lambda}_{k,2} &=& {\displaystyle \varphi_{k+2}^T(x,y)\left( D_{k+2,2}^{(k)} \right)^T \frac{\partial}{\partial y} \varphi_k(x,y)  - \varphi_{k}^T(x,y)D_{k+2,2}^{(k)}\frac{\partial}{\partial y} \varphi_{k+2}(x,y),}
\\[7pt]
\tilde{\Omega}_{k,1} &=& {\displaystyle \varphi_{k+1}^T(x,y)\left( D_{k+1,1}^{(k)} \right)^T \frac{\partial}{\partial x} \varphi_k(x,y) -  \varphi_{k}^T(x,y)D_{k+1,1}^{(k)}\frac{\partial}{\partial x} \varphi_{k+1}(x,y),}
\\[7pt]
\tilde{\Omega}_{k,2} &=& {\displaystyle \varphi_{k+1}^T(x,y)\left( D_{k+1,2}^{(k)} \right)^T \frac{\partial}{\partial y} \varphi_k(x,y) -  \varphi_{k}^T(x,y)D_{k+1,2}^{(k)}\frac{\partial}{\partial y} \varphi_{k+1}(x,y)}.
\end{array}$$
\end{corollary}

\begin{proof}
Since $\varphi_{s}^{T}(x_1,y_1)\left( D_{s,i}^{(k)} \right)^{T} \varphi_k(x_1,y_1)$ is a scalar function for $s \in \{k+1,k+2\}$ and $i \in \{1,2\}$, we can write (compare with (\ref{LambdaOmegaki}))
$$\begin{aligned}
\Lambda_{k,i} &=\varphi_{k+2}^T(x_1,y_1)\left( D_{k+2,i}^{(k)} \right)^T \left[\varphi_k(x_2,y_2) - \varphi_k(x_1,y_1) \right]  - \varphi_{k}^T(x_1,y_1)D_{k+2,i}^{(k)}\left[\varphi_{k+2}(x_2,y_2) - \varphi_{k+2}(x_1,y_1) \right], \\[5pt]
\Omega_{k,i} &=\varphi_{k+1}^T(x_1,y_1)\left( D_{k+1,i}^{(k)} \right)^T \left[ \varphi_k(x_2,y_2) - \varphi_k(x_1,y_1) \right]  -  \varphi_{k}^T(x_1,y_1)D_{k+1,i}^{(k)}\left[ \varphi_{k+1}(x_2,y_2) - \varphi_{k+1}(x_1,y_1)  \right].
\end{aligned}
$$
Also, ${\displaystyle \left( x_1 + \frac{1}{x_1}\right) - \left( x_2 + \frac{1}{x_2}\right) = \left( x_1 - x_2 \right)\frac{x_1x_2-1}{x_1x_2}}$, so if ${\displaystyle x_1 + \frac{1}{x_1} \neq x_2 + \frac{1}{x_2}}$,
$$
\begin{array}{l}
{\mathcal K}_n(x_1,y_1,x_2,y_2)={\displaystyle \frac{\Omega_{n,1} + \Lambda_{n,1} + \Lambda_{n-1,1}}{\left( x_1 + \frac{1}{x_1}\right) - \left( x_2 + \frac{1}{x_2}\right)}=\frac{x_1x_2}{x_1x_2-1}\times \left[ \right.}
\\[20pt]
\quad {\displaystyle \varphi_{n+1}^T(x_1,y_1)\left( D_{n+1,1}^{(n)} \right)^T \cdot \frac{\varphi_n(x_2,y_2) - \varphi_n(x_1,y_1)}{x_1-x_2} }
\\[20pt]
\quad {\displaystyle - \; \varphi_{n}^T(x_1,y_1)D_{n+1,1}^{(n)}\cdot \frac{ \varphi_{n+1}(x_2,y_2) - \varphi_{n+1}(x_1,y_1)}{x_1-x_2}  }
\\[20pt]
\quad {\displaystyle + \; \varphi_{n+2}^T(x_1,y_1)\left( D_{n+2,1}^{(n)} \right)^T \cdot \frac{\varphi_n(x_2,y_2) - \varphi_n(x_1,y_1)}{x_1-x_2} }
\\[20pt]
\quad {\displaystyle - \; \varphi_{n}^T(x_1,y_1)D_{n+2,1}^{(n)}\cdot \frac{\varphi_{n+2}(x_2,y_2) - \varphi_{n+2}(x_1,y_1)}{x_1-x_2}  }
\\[20pt]
\quad {\displaystyle  + \; \varphi_{n+1}^T(x_1,y_1)\left( D_{n+1,1}^{(n-1)} \right)^T \cdot \frac{\varphi_{n-1}(x_2,y_2) - \varphi_{n-1}(x_1,y_1)}{x_1-x_2}}
\\[20pt]
\quad {\displaystyle \left. -  \; \varphi_{n-1}^T(x_1,y_1)D_{n+1,1}^{(n-1)}\cdot\frac{\varphi_{n+1}(x_2,y_2) - \varphi_{n+1}(x_1,y_1)}{x_1-x_2} \right]}.
\end{array}
$$
The first equality follows letting $(x_2,y_2) \rightarrow (x_1,y_1)=(x,y)$ and the second one is obtained in a similar way.
\end{proof}

\section{A connection with the one variable case}\label{Sec1variable}

Consider the rectangle ${\mathcal R}=[a,b]\times [c,d]$, $0<a<b<\infty$, $0 < c < d < \infty$, and a positive Borel measure on ${\mathcal R}$ that can be factorized in the form $d\mu(x,y)=d\mu_1(x)d\mu_2(y)$. Let $\langle \cdot , \cdot \rangle_{\mu}$ be the inner product given by (\ref{inner}) and $\langle \cdot , \cdot \rangle_{\mu_i}$ the corresponding inner products for the measures $d\mu_i$, $i=1,2$:
$$\begin{array}{lcl}
{\displaystyle \langle f,g \rangle_{\mu_1}=\int_a^b f(x)g(x)d\mu_1(x), }&\quad& f,g \in L_2^{\mu_1}=\left\{ h: [a,b] \rightarrow \RR \;:\;  {\displaystyle \int_a^b h^2(x)d\mu_1(x) < \infty} \right\}, \\[15pt]
{\displaystyle \langle f,g \rangle_{\mu_2}=\int_c^d f(y)g(y)d\mu_2(y),} &\quad& f,g \in L_2^{\mu_2}=\left\{ h: [c,d] \rightarrow \RR \;:\;  {\displaystyle \int_c^d h^2(y)d\mu_2(y)} < \infty \right\}.
\end{array}$$
Notice that the corresponding moments $m_k^{(i)}$ are strictly positive, for all $k \in \ZZ$ and $i\in \{ 1,2 \}$. Let us denote by $\left\{ \psi_n^{(i)} \right\}_{n \geq 0}$ for $i=1,2$ the families of orthogonal Laurent polynomials in one variable with respect to me measures $\mu_i$ and the ``balanced'' ordering (\ref{balanced1v}). Thus we can prove how from these two families we can construct an orthonormal basis of Laurent polynomials in two variables.

\begin{proposition}\label{aconnection}
Under the above conditions, let $\varphi_{n,k}$ be given by
$$\varphi_{n,k}(x,y)=\psi_{n-k}^{(1)}(x)\psi_k^{(2)}(y), \quad  \textrm{for all} \quad n\geq 0 \quad \textrm{and} \quad k \in \{0,1,\ldots,n\}.$$
Then the set $\left\{ \varphi_{n,k} : n\geq 0, \;k=0,\ldots,n  \right\}$ is an orthonormal basis of ${\mathcal L}$.
\end{proposition}

\begin{proof}
Recall ${\mathcal L}_k=\Span \left\{ \phi_0,\ldots,\phi_k \right\}$. It is clear from the construction of the ordering (\ref{matrixL}) and (\ref{L})-(\ref{phik}) that $\psi_{n-k}^{(1)}(x) \in {\mathcal L}_{n-k} \backslash {\mathcal L}_{n-k-1}$, $\psi_{k}^{(2)}(y) \in {\mathcal L}_{k} \backslash {\mathcal L}_{k-1}$ and $\psi_{n-k}^{(1)}(x)\psi_k^{(2)}(y) \in  {\mathcal L}_{n}\backslash {\mathcal L}_{n-1}$.
Thus, for fixed $n\geq 0$ and $k,l \in \{0,\ldots,n \}$, it is clear from Fubini's theorem that
$$\langle \psi_{n-k}^{(1)}(x)\psi_k^{(2)}(y) , \psi_{n-l}^{(1)}(x)\psi_l^{(2)}(y) \rangle_{\mu} =\langle \psi_{n-k}^{(1)}(x) , \psi_{n-l}^{(1)}(x) \rangle_{\mu_1} \cdot \langle \psi_{k}^{(2)}(y) , \psi_{l}^{(2)}(y) \rangle_{\mu_2}=\delta_{k,l}.$$
Also for $n \neq m$, $n,m \geq 0$, $k \in \{0\ldots,n\}$ and $l \in \{0\ldots,m\}$, we get by the same reason
$$\langle \psi_{n-k}^{(1)}(x)\psi_k^{(2)}(y) , \psi_{m-l}^{(1)}(x)\psi_l^{(2)}(y) \rangle_{\mu} = \langle \psi_{n-k}^{(1)}(x) , \psi_{m-l}^{(1)}(x) \rangle_{\mu_1} \cdot \langle \psi_{k}^{(2)}(y) , \psi_{l}^{(2)}(y) \rangle_{\mu_2}=0.$$
This concludes the proof.
\end{proof}

The aim of this section is to make use of Theorem \ref{recu1var} and Proposition \ref{aconnection} to obtain explicitly the relations \eqref{5TR} in this particular situation. We start with the following
\begin{lemma}\label{onevariable}
Let $\{ \Omega_n^{(2)} \}_{n \geq 0}$ and $\{ C_n^{(2)} \}_{n \geq 0}$ be the sequences of positive real numbers appearing in Theorem \ref{recu1var}, associated with the measure $d\mu_2$. Then, under the above conditions the family $\{ \varphi_n \}_{n \geq 0}$ satisfies the recurrence realations
\begin{equation}\label{onevariable1}
\begin{aligned}
&C_{2m}^{(2)}\varphi_{n,2m+1}(x,y)=\left( \Omega_{2m}^{(2)}y-1 \right)\varphi_{n-1,2m}(x,y)-C_{2m-1}^{(2)}\varphi_{n-2,2m-1}(x,y), \quad 0 \leq 2m-1 \leq n-2,
\\[8pt]
&C_{2m+1}^{(2)}\varphi_{n+1,2m+2}(x,y)=\left(1-\frac{\Omega_{2m+1}^{(2)}}{y} \right)\varphi_{n,2m+1}(x,y)-C_{2m}^{(2)}\varphi_{n-1,2m}(x,y), \quad 0 \leq 2m \leq n-1,
\\[8pt]
&C_0^{(2)}\varphi_{1,1}(x,y)=\left( \Omega_0^{(2)}y-1\right)\varphi_{0,0}(x,y).
\end{aligned}
\end{equation}
\end{lemma}

\begin{proof}
From Theorem \ref{recu1var} we have $$C_{2m}^{(2)}\psi_{2m+1}^{(2)}(y)=\left( \Omega_{2m}^{(2)}y-1 \right)\psi_{2m}^{(2)}(y)-C_{2m-1}^{(2)}\psi_{2m-1}^{(2)}(y),$$ and multiplying in both sides of this equality by $\psi_{n-(2m+1)}^{(1)}(x)$ we get from Proposition \ref{aconnection} the first relation in (\ref{onevariable1}). We can prove the second and third relations in (\ref{onevariable1}) proceeding in a similar way.
\end{proof}

\begin{remark}
A similar result can be proved involving only the coefficients $\{ \Omega_n^{(1)} \}_{n \geq 0}$ and $\{ C_n^{(1)} \}_{n \geq 0}$ related to the measure $d\mu_1$, we omit these details. It should be clear to the reader that despite the recurrence in Lemma \ref{onevariable} only involves the coefficients related to the measure $d\mu_2$, there is no relation between the families $\{ \varphi_n \}_{n \geq 0}$ and $\{ \tilde{\varphi}_n \}_{n \geq 0}$ associated with two measures of the form $d\mu(x,y)=d\mu_1(x)d\mu_2(y)$ and $d\tilde{\mu}(x,y)=d\tilde{\mu}_1(x)d\mu_2(y)$ respectively, since the influence of the first measure is due to
\begin{equation}\label{influence}
\begin{array}{lcl}\varphi_{0,0} \equiv \frac{1}{\sqrt{m_0^{(1)}m_0^{(2)}}}, && \\ \\ \varphi_{n,0}(x,y)=\frac{1}{\sqrt{m_0^{(2)}}}\psi_n^{(1)}(x), &\;& \varphi_{n,1}(x,y)=\frac{1}{C_0^{(2)}\sqrt{m_0^{(2)}}}\left( \Omega_0^{(2)}y - 1 \right)\psi_{n-1}^{(1)}(x), \quad \forall n \geq 1.
\end{array}
\end{equation}
From Lemma \ref{onevariable} we see actually how the combination of (\ref{influence}) and the relations in (\ref{onevariable1}) let us compute the full sequence $\{ \varphi_n \}_{n\geq 0}$.
\end{remark}

Next, let us see how explicit expressions for the matrices ${\mathcal F}_1$ and ${\mathcal F}_2$ in (\ref{representation1x})-(\ref{representation2x}) and (\ref{representation1y}) respectively, can be found from Lemma \ref{onevariable}. We present a proof for ${\mathcal F}_2$, the corresponding for ${\mathcal F}_1$ follows in a similar way.

\begin{theorem}
Under the above conditions, for $l \geq 0$,  let us introduce the constants
$$
\Gamma_l=\frac{C_l^{(2)}C_{l+1}^{(2)}}{\Omega_{l+1}^{(2)}}>0, \quad \Delta_l=(-1)^lC_l^{(2)}\left( \frac{1}{\Omega_{l}^{(2)}} - \frac{1}{\Omega_{l+1}^{(2)}}\right), \quad \Xi_l=\Omega_l^{(2)}+\frac{1}{\Omega_l^{(2)}} + \frac{\left( C_l^{(2)} \right)^2}{\Omega_{l+1}^{(2)}} + \frac{\left( C_{l-1}^{(2)} \right)^2}{\Omega_{l-1}^{(2)}}>0,
$$
with $C_{-1}^{(2)}=0$ and $\Omega_{-1}^{(2)}$ an arbitrary nonzero constant. Then, the matrix ${\mathcal F}_2$ in (\ref{representation1y}) is explicitly given for all $n \geq 1$ by
$$\begin{array}{ccl} D_{n+1,2}^{(n-1)} &=& \left( {\mathcal O}_{n,2} \;|\; \mathrm{diag}\left( \Gamma_0, \Gamma_1,\ldots, \Gamma_{n-1}  \right) \right), \\ D_{n,2}^{(n-1)} &=& \left( {\mathcal O}_{n,1} \;|\; \mathrm{diag}\left( \Delta_0, \Delta_1,\ldots, \Delta_{n-1}  \right) \right), \\
D_{n-1,2}^{(n-1)} &=& \mathrm{diag}\left( \Xi_0 ,\ldots, \Xi_{n-1}  \right).\end{array}$$
The matrices $D_{n-2,2}^{(n-1)}$ ($n \geq 2$) and $D_{n-3,2}^{(n-1)}$ ($n \geq 3$) are the transpose of $D_{n-1,2}^{(n-2)}$ and $D_{n-1,2}^{(n-3)}$.
\end{theorem}

\begin{proof}
The initial conditions $D_{s,2}^{(n)}$ with $n=0$ ($s=0,1,2$) and $n=1$ ($s=0,1,2,3$) are deduced by direct computations derived from Theorem \ref{recu1var} and Proposition \ref{aconnection}, in an analog procedure as in (\ref{initialprocedure0})-(\ref{initialprocedure1}).

For $n \geq 2$ we have to consider separately the cases that involves the two-term relation for $\psi_0^{(2)}$ and $\psi_1^{(2)}$ in Theorem \ref{recu1var}. So, from Theorem \ref{recu1var} and Proposition \ref{aconnection} we can write
\begin{equation}\label{onevariable1ini}
C_0^{(2)}\varphi_{n,1}=\left(\Omega_0^{(2)}y-1\right)\varphi_{n-1,0} \quad \Rightarrow \quad \frac{1}{y}\varphi_{n-1,0}=\Omega_0^{(2)}\varphi_{n-1,0}- C_0^{(2)}\frac{1}{y}\varphi_{n,1}
\end{equation}
and
\begin{equation}\label{onevariable2ini}C_1^{(2)}\varphi_{n+1,2}=\left(1-\frac{\Omega_1^{(2)}}{y}\right)\varphi_{n,1}-C_0^{(2)}\varphi_{n-1,0} \;\; \Rightarrow \;\; \frac{1}{y}\varphi_{n,1}=\frac{1}{\Omega_1^{(2)}}\varphi_{n,1}-\frac{C_1^{(2)}}{\Omega_1^{(2)}}\varphi_{n+1,2}-\frac{C_0^{(2)}}{\Omega_1^{(2)}}\varphi_{n-1,0}.
\end{equation}
If we substitute in (\ref{onevariable1ini}) the term $\frac{1}{y}\varphi_{n,1}$ in (\ref{onevariable2ini}) and we add $y\varphi_{n-1,0}$ we get
$$\left( y + \frac{1}{y} \right)\varphi_{n-1,0} = \Xi_0\varphi_{n-1,0}+\Delta_0\varphi_{n,1} + \Gamma_0 \varphi_{n+1,2}.$$
In a similar way, we can write (\ref{onevariable2ini}) with $n$ replaced by $n-1$ as
\begin{equation}\label{onevariable3ini}
y\varphi_{n-1,1}=C_1^{(2)}y\varphi_{n,2}+\Omega_1^{(2)}\varphi_{n-1,1}+C_0^{(2)}y\varphi_{n-2,0}
\end{equation}
and we also have from Theorem \ref{recu1var} and Proposition \ref{aconnection} the relation
\begin{equation}\label{onevariable4ini}
y\varphi_{n,2}=\frac{C_2^{(2)}}{\Omega_2^{(2)}}\varphi_{n+1,3}+\frac{1}{\Omega_2^{(2)}}\varphi_{n,2}+\frac{C_1^{(2)}}{\Omega_2^{(2)}}\varphi_{n-1,1}.
\end{equation}
If we substitute in (\ref{onevariable3ini}) the terms $y\varphi_{n,2}$ in (\ref{onevariable4ini}) and $y\varphi_{n-2,0}$ in (\ref{onevariable1ini}) with $n$ replaced by $n-1$ we get
$$\left( y + \frac{1}{y} \right)\varphi_{n-1,1} = \Gamma_1\varphi_{n+1,3} + \Delta_1\varphi_{n,2} + \Xi_1\varphi_{n-1,1} + \Delta_0\varphi_{n-2,0}.$$

For the general case $n \geq 3$ and $l=2,3,\ldots,n-1$ we start writing the first equation of (\ref{onevariable1}) as
\begin{equation}\label{onevariable1b}
\frac{1}{y}\varphi_{n-1,2m}(x,y) = -C_{2m}^{(2)}\frac{1}{y}\varphi_{n,2m+1}(x,y)+\Omega_{2m}^{(2)}\varphi_{n-1,2m}(x,y)-C_{2m-1}^{(2)}\frac{1}{y}\varphi_{n-2,2m-1}(x,y)
\end{equation}
and the second equation of (\ref{onevariable1}) as
\begin{equation}\label{onevariable2b}
\begin{aligned}
\frac{1}{y}\varphi_{n,2m+1}(x,y) = {\displaystyle \frac{1}{\Omega_{2m+1}^{(2)}}} & \left[ -C_{2m+1}^{(2)}\varphi_{n+1,2m+2}(x,y) +\varphi_{n,2m+1}(x,y)-C_{2m}^{(2)}\varphi_{n-1,2m}(x,y)\right], 
\\[10pt]
\frac{1}{y}\varphi_{n-2,2m-1}(x,y) =  \frac{1}{\Omega_{2m-1}^{(2)}} & \left[ -C_{2m-1}^{(2)}\varphi_{n-1,2m}(x,y) \right.
\\[10pt]
& \left. +\varphi_{n-2,2m-1}(x,y)-C_{2m-2}^{(2)}\varphi_{n-3,2m-2}(x,y)\right].
\end{aligned}
\end{equation}
Thus, if we substitute (\ref{onevariable2b}) in (\ref{onevariable1b}) and we add the term $y\varphi_{n-1,2m}(x,y)$ from the first equation of (\ref{onevariable1}) we get
$$
\begin{aligned}
\left( y + \frac{1}{y} \right)\varphi_{n-1,2m}(x,y) = &\Gamma_{2m}\varphi_{n+1,2m+2}(x,y) 
 + \Delta_{2m}\varphi_{n,2m+1}(x,y)+\Xi_{2m}\varphi_{n-1,2m}(x,y) 
 \\
&+\Delta_{2m-1}\varphi_{n-2,2m-1}(x,y)+\Gamma_{2m-2}\varphi_{n-3,2m-2}(x,y).
\end{aligned}
$$
A similar relation is obtained for $\left( y + \frac{1}{y} \right)\varphi_{n-1,2m-1}(x,y)$, yielding for all $l=2,3,\ldots,n-1$ and $n \geq 3$,
$$
\begin{aligned}
\left( y + \frac{1}{y} \right)\varphi_{n-1,l}(x,y) = &\Gamma_{l}\varphi_{n+1,l+2}(x,y) 
+ \Delta_{l}\varphi_{n,l+1}(x,y)+\Xi_l\varphi_{n-1,l}(x,y)
\\
& + \Delta_{l-1}\varphi_{n-2,l-1}(x,y)+\Gamma_{l-2}\varphi_{n-3,l-2}(x,y).
\end{aligned}
$$
\end{proof}

As it is indicated in \cite[Section 2]{CDM4}, the families $\{ \psi_k \}_{k\geq 0}$ of orthogonal Laurent polynomials in the one variable case computed from Theorem \ref{recu1var} are related to the families of ordinary polynomials satisfying Laurent orthogonal conditions (that have been considered in the literature, e.g. by A. Sri Ranga and collaborators, see \cite{ranga1,ranga2,ranga3,ranga4}). So, we can make use of some of the results available in those references to get explicit expressions for the coefficients $\{ \Omega_n \}_{n \geq 0}$ and $\{ C_n \}_{n \geq 0}$ related to some particular absolutely continuous measures, like
$$
\begin{array}{lllll}
{\displaystyle d\omega_1(x)=\frac{dx}{\sqrt{(b-x)(x-a)}},} & &  {\displaystyle d\omega_2(x)=\frac{dx}{\sqrt{x}},} 
 \\[15pt]
 {\displaystyle d\omega_3^{\mu}(x)=\frac{[(b-x)(x-a)]^{\mu-\frac{1}{2}}}{\left( \sqrt{b} - \sqrt{a} \right) x^{\mu}}dx,}
 & & 
{\displaystyle d\omega_4(x)=\frac{x\left( 1 + \frac{\sqrt{ab}}{x} \right)^2}{\sqrt{(b-x)(x-a)}}dx, \;\;} 
 \\[15pt]
 {\displaystyle d\omega_5(x)=\frac{dx}{\left( x + \sqrt{ab} \right)\sqrt{(b-x)(x-a)}},\;\;} & &{\displaystyle d\omega_6^{\kappa}(t)=\frac{1}{2\kappa \sqrt{\pi}}\left( 1+\frac{1}{t} \right)e^{-\left( \frac{\log(t)}{2\kappa} \right)^2}}dt,
\end{array}
$$
with $x \in (a,b)$, $0<a<b<\infty$, $t>0$, $\kappa>0$ and $\mu > -1/2$. Thus, if the measure $d\mu(x,y)=d\mu_1(x)d\mu_2(y)$ defined on the rectangle ${\mathcal R}$ is of the form $\mu_1,\mu_2 \in \{ \omega_i \}_{i=1}^{6}$ we can recover directly from the results of this section the recurrence relations for $\{ \varphi_n \}_{n \geq 0}$ explicitly.

\section{Applications for future research}\label{Sec-App}

We conclude this paper with two applications to the {\em Theory of Orthogonal Laurent Polynomials of Two Real Variables} that we have introduced for future research.

\subsection{Applications to cubature formulas}\label{Appl1}

Let us suppose that we are interested in the numerical estimation of integrals of the form $I_{\mu}(f)=\iint_{D} f(x,y)d\mu(x,y)$, being $\mu$ a positive Borel measure supported on $D\subset \RR^2$ such that $\{ x=0 \}\cup \{ y=0 \} \not\in D$. The multi-dimensional analogue of the well known {\em Gaussian quadrature formulas} for the one-dimensional case are called {\em Gaussian cubature formulas}, see e.g. \cite{Xu} (Sections 3.7-3.8), \cite{cools,harris,lasse} and references therein. These rules are of the form
\begin{equation}\label{qf}
I_n(f)=\sum_{i=1}^{n} \lambda_i f(x_i), \quad \quad x_i \subset D, \; x_i \neq x_j \;\;\textrm{for all} \;\;i \neq j, \quad i,j=1,\ldots,n,
\end{equation}
where $\{x_i\}_{i=1}^{n}$ and $\{\lambda_i\}_{i=1}^{n}$ are called the set of nodes and weights (or coefficients), respectively. As it was mentioned in the Introduction, it has been numerically proven in the literature that when the integrand presents singularities near the subset of integration, quadrature formulas based on Laurent polynomials are preferable compared to the classical Gaussian rules. Having this in mind and from the results presented in this paper, it results of interest to consider new techniques of numerical integration based on Laurent polynomials of two (or more) real variables, which we could coin as {\em L-Gaussian cubature formulas}.

The cubature rule $I_n(f)$ has (exact) degree of precision $s$ if $I_{\mu}(P)=I_n(P)$ for all polynomials $P$ of degree less than or equal to $s$, and it is not exact for some polynomial of degree $s+1$. In the one-dimensional case, quadrature formulas are usually constructed of interpolatory-type: for a fixed set of $n$ nodes, the rule is obtained by integrating exactly the unique interpolatory polynomial to $f$ at these nodes in its Lagrange form, and the degree of precision of the resulting rule is at least $n-1$. The set of nodes can be adequately fixed to increase the degree of precision up to $2n-1$ (Gaussian quadrature formula, that always exist and it is unique). The set of nodes in this case are the zeros of a $n$-th orthogonal polynomial with respect to $\mu$, and the corresponding weights are positive, that is of interest due to convergence and stability reasons. All these ideas have been translated successfully to the Laurent case (see e.g. \cite{CDM2}), where the rules are obtained by looking for exactness in subspaces of Laurent polynomials (in particular, the balanced ones).

As indicated in the survey published by Cools et al. \cite{cools} in 2001 about cubature formulae, in more than one dimension things look worse and there are more questions than answers. However, some progress have been made, and several results of a certain generality have been found. The algebraic manifold generated by a polynomial $P(x,y)$ of exact degree $s$ is the set $\left\{ (x,y)\in \CC^2 : P(x,y)=0 \right\}$. A cubature formula (\ref{qf}) of degree of precision $s$ is interpolatory if the nodes do not lie on an algebraic manifold of degree $s$ and the coefficients are uniquely determined by the nodes (integrals of Lagrange Laurent polynomials). The number of nodes $n$ should be in this case the number of linearly independent polynomials of degree less than or equal to $s$ that, in the bi-variate case, is $\frac{1}{2}(s+1)(s+2)$; but in more than one dimension, $n$ might be lower since some of the weights may vanish. This motivates the concept of interpolatory minimal cubature rule: for fixed $s$, $n$ is minimal.

In the $k$-dimensional case, by a Gaussian cubature formula (see Section 3.8 of \cite{Xu}) we understand a rule having odd degree of precision $2p-1$ ($p\in \NN$) and minimal number of nodes: ${p+k-1\choose k}$. From 1970 until 2012, the only necessary and sufficient result available in the literature for the existence of a $k$-dimensional Gaussian cubature formula of degree $2p-1$  was due to Mysovskikh \cite{Mys}: the rule exists, if and only if, the orthogonal polynomials of degree $p$ have exactly ${p+k-1\choose k}$ common zeros (nodes of the cubature). Similar to the one-dimensional case, these common zeros can be computed as all joint eigenvalues of $k$ certain associated Jacobi matrices, see  Theorem 3.7.2 in \cite{Xu}. The main drawback of this characterization is that it is known that in many cases, orthogonal polynomials does not have distinct common zeros, and only for a very few particular situations it is known that they share common zeros (see Corollary 3.7.7 and Section 5.4 in \cite{Xu}). Inspired on these negative known results, we think that this way does not seems appropriate for the construction of cubature rules based on Laurent polynomials.

An alternative criterion was obtained in 2012 by Lasserre \cite{lasse}: a $k$-dimensional Gaussian cubature formula of degree $2p-1$ exists, if and only if, a certain overdetermined linear system of equations has a solution. The coefficient matrix of this linear system comes from expressing the product of two orthonormal polynomials $P_1,P_2$ of degree $p$ in the basis of orthonormal polynomials of degree up to $2p$. So, it is an open question if Theorem 3.1 in \cite{lasse} can be adapted to the orthonormal basis of Laurent polynomials that we have introduced in this paper. We should note in this regard a key difference in our case, which is the fact that if $L_s,\tilde{L}_s \in {\mathcal L}_{s} \backslash {\mathcal L}_{s-1}$, then $L_s\cdot \tilde{L}_s \in {\mathcal L}_{2s+2}$. This property for Laurent polynomials, which does not hold obviously for ordinary polynomials, can be directly obtained from the relations
\begin{equation}\label{crelations}
c_{2s}+c_{2t}=c_{2(s+t)}, \quad c_{2s-1}+c_{2t-1}=c_{2(s+t)-1}, \quad c_{2s}+c_{2t-1}=\left\{ \begin{array}{crl} c_{2(t-s)-1} &\textrm{if} & t-s>0, \\ c_{2(s-t)} &\textrm{if} & t-s \leq 0. \end{array} \right.
\end{equation}
where we recall that the sequence $\{c_n\}_{n\ge 0}$ is defined in \eqref{defc}.

In 2015, a further characterization for the existence of a $k$-dimensional Gaussian cubature formula of degree $2p-1$ when the nodes of the formula have Lagrange polynomials of degree at most $p$, was given by Harris \cite{harris}. The main condition is that the Lagrange polynomial at each node is a scalar multiple of the reproducing kernel of degree $p-1$ evaluated at the nodes plus an orthogonal polynomial of degree $p$. Stronger conditions are given for the case where the cubature is exact for polynomials of degree up to $2p$. In that paper, two particular Gaussian cubature rules are constructed, based on Geronimus and Morrow-Patterson nodes. We believe that this third approach could be another way to characterize L-Gaussian cubature formulas. Actually, we can extend to the Laurent context the main result of Harris paper (Lemma 1 in \cite{harris}), that is an equivalence between a cubature formula and a formula for the Lagrange polynomials for the nodes (see Lemma \ref{Harris} further). We need first the following technical result, that in the ordinary polynomial situation is trivial.
\begin{lemma}\label{Harrisprev}
Every Laurent polynomial $L \in {\mathcal L}_{2p-1}$ ($p \in \NN$, $p\geq 2$) is a linear combination of monomials of the form $l_1\cdot l_2$, with $l_1 \in {\mathcal L}_{p-1}$ and $l_2 \in {\mathcal L}_{p}$.
\end{lemma}
\begin{proof}
We proceed by induction over $p$. It is immediate to prove the result for $p=2$ and $p=3$. If we suppose that the property holds for every Laurent polynomial in ${\mathcal L}_{2p-3}$, since every Laurent polynomial in ${\mathcal L}_{2p-1}$ is a linear combination of $\phi_{2p-1}$, $\phi_{2p-2}$ and $\psi \in {\mathcal L}_{2p-3}$, it is clear that the proof follows if we prove the result for $\phi_s$, $s \in \{2p-2,2p-1\}$. Actually, what we need is to prove it for the first $E\left[ \frac{s}{2} \right]+1$ components of the vectors $\phi_s$, otherwise the proof follows from the same reasoning by interchanging the roles of the variables $x$ and $y$.

Let us start with the case $s=2p-1$. We analyze the $(t+1)$-th component of the vector $\phi_{2p-1}$ by using the first two relations in (\ref{crelations}) and by distinguishing two cases:
\begin{itemize}
\item Case $p=2k$. If $t=2\nu$ with $0 \leq \nu \leq k-1$, then this monomial is given by
$$x^{c_{4k-1-t}}y^{c_{t}}=x^{c_{2(k+k-\nu)-1}}y^{c_{2\nu}}=\underbrace{x^{c_{2k-1}}}_{\in \phi_{p-1}}\cdot \underbrace{x^{c_{2(k-\nu)-1}}y^{c_{2\nu}}}_{\in \phi_{p-1}},$$
whereas if $t=2\nu+1$ with $0 \leq \nu \leq k-1$, then it is given by
$$x^{c_{4k-1-t}}y^{c_{t}}=x^{c_{2(k+k-1-\nu)}}y^{c_{2\nu+1}}=\underbrace{x^{c_{2k}}}_{\in \phi_{p}}\cdot \underbrace{x^{c_{2(k-1-\nu)}}y^{c_{2\nu+1}}}_{\in \phi_{p-1}}.$$
\item Case $p=2k+1$. If $t=2\nu$ with $0 \leq \nu \leq k$, then the corresponding component of the vector $\phi_{2p-1}$ is given by
$$x^{c_{4k+1-t}}y^{c_{t}}=x^{c_{2(k+1+k-\nu)-1}}y^{c_{2\nu}}=\underbrace{x^{c_{2(k+1)-1}}}_{\in \phi_{p}}\cdot \underbrace{x^{c_{2(k-\nu)-1}}y^{c_{2\nu}}}_{\in \phi_{p-2}},$$
whereas if $t=2\nu+1$ with $0 \leq \nu \leq k-1$, then it is given by
$$x^{c_{4k+1-t}}y^{c_{t}}=x^{c_{2(k+k-\nu)}}y^{c_{2\nu+1}}=\underbrace{x^{c_{2k}}}_{\in \phi_{p-1}}\cdot \underbrace{x^{c_{2(k-\nu)}}y^{c_{2\nu+1}}}_{\in \phi_{p}}.$$
\end{itemize}
The case $s=2p-2$ is similar. We analyze the $(t+1)$-th component of the vector $\phi_{2p-2}$ again from the first two relations in (\ref{crelations}) and by distinguishing two cases:
\begin{itemize}
\item Case $p=2k$. If $t=2\nu$ with $0 \leq \nu \leq k-1$, then this monomial is given by
$$x^{c_{4k-2-t}}y^{c_{t}}=x^{c_{2(k-1+k-\nu)}}y^{c_{2\nu}}=\underbrace{x^{c_{2(k-1)}}}_{\in \phi_{p-2}}\cdot \underbrace{x^{c_{2(k-\nu)}}y^{c_{2\nu}}}_{\in \phi_{p}},$$
whereas if $t=2\nu+1$ with $0 \leq \nu \leq k-1$, then it is given by
$$x^{c_{4k-2-t}}y^{c_{t}}=x^{c_{2(k-1+k-\nu)-1}}y^{c_{2\nu+1}}=\underbrace{x^{c_{2(k-1)-1}}}_{\in \phi_{p-3}}\cdot \underbrace{x^{c_{2(k-\nu)-1}}y^{c_{2\nu+1}}}_{\in \phi_{p}}.$$
\item Case $p=2k+1$. If $t=2\nu$ with $0 \leq \nu \leq k$, then the corresponding component of the vector $\phi_{2p-2}$ is given by
$$x^{c_{4k-t}}y^{c_{t}}=x^{c_{2(k+k-\nu)}}y^{c_{2\nu}}=\underbrace{x^{c_{2k}}}_{\in \phi_{p-1}}\cdot \underbrace{x^{c_{2(k-\nu)}}y^{c_{2\nu}}}_{\in \phi_{p-1}},$$
whereas if $t=2\nu+1$ with $0 \leq \nu \leq k-1$, then it is given by
$$x^{c_{4k-t}}y^{c_{t}}=x^{c_{2(k+k-\nu)-1}}y^{c_{2\nu+1}}=\underbrace{x^{c_{2k-1}}}_{\in \phi_{p-2}}\cdot \underbrace{x^{c_{2(k-\nu)-1}}y^{c_{2\nu+1}}}_{\in \phi_{p-1}}.$$
\end{itemize}
This concludes the proof
\end{proof}

\begin{lemma}\label{Harris}
Suppose that $\{ \lambda_i \}_{i=1}^{n}$ are real numbers, $\{ (x_i,y_i) \}_{i=1}^{n} \subset D$ is a set of $n$ distinct points, $p\geq 1$ and let ${\mathcal K}_{p-1}$ be the $(p-1)$-th reproducing kernel \eqref{eq-ker}. Define
$$
{\mathcal S}_p=\left\{ \eta \in {\mathcal L}_{p} \;:\; \langle \eta,\chi \rangle \;\textrm{for all}\; \chi \in {\mathcal L}_{p-1} \right\}
$$
and suppose that a corresponding set of Lagrange Laurent polynomials $\{ \zeta_i(x,y) \}_{i=1}^{n}$ exist: $\zeta_i \in {\mathcal L}_p$, $\zeta_i(x_j,y_j)=\delta_{i,j}$, $1 \leq i,j \leq n$.

Suppose that for all $\psi \in {\mathcal L}_{p}$, there exists $\eta \in {\mathcal S}_p$ such that
\begin{equation}\label{Harris1}
\psi(x,y)=\sum_{i=1}^{n} \psi(x_i,y_i)\zeta_i(x,y)+\eta(x,y)
\end{equation}
with
\begin{equation}\label{Harris2}
\zeta_i(x,y)=\lambda_i{\mathcal K}_{p-1}(x,y,x_i,y_i)+\eta_i(x,y), \quad \eta_i \in {\mathcal S}_p, \quad i=1,\ldots,n.
\end{equation}
Then
\begin{equation}\label{Harris3}
I_{\mu}(\psi)=\iint_{D} \psi(x,y)d\mu(x,y) = I_n(\psi)=\sum_{i=1}^{n} \lambda_i \psi(x_i,y_i), \quad \textrm{for all} \; \psi \in {\mathcal L}_{2p-1}.
\end{equation}
\end{lemma}
\begin{proof}
Let $\psi_1 \in {\mathcal L}_{p-1}$. From the reproducing property, \eqref{Harris2} and $\langle \psi_1 , \eta_i \rangle=0$, we have for all $i=1,\ldots,n$ that
$$
\lambda_i \psi_1(x_i,y_i)=\langle \psi_1 , \lambda_i{\mathcal K}_{p-1}(\cdot,\cdot,x_i,y_i)\rangle = \langle \psi_1 , \zeta_i - \eta_i \rangle = \langle \psi_1 , \zeta_i \rangle.
$$
If $\psi_2 \in {\mathcal L}_{p}$ it follows from \eqref{Harris1} that
$$
\langle \psi_1 , \psi_2 \rangle=\sum_{i=1}^{n} \psi_2(x_i,y_i)\langle \psi_1,\zeta_i \rangle + \langle \psi_1,\eta\rangle = \sum_{i=1}^{n} \lambda_i \psi_1(x_i,y_i)\psi_2(x_i,y_i).
$$
So, \eqref{Harris3} holds for functions of the form $\psi_1 \psi_2$ and from Lemma \ref{Harrisprev}, it holds for all $\psi \in {\mathcal L}_{2p-1}$.
\end{proof}

It is a complicated problem to determine when a family of Lagrange Laurent polynomials exists, but nevertheless, this Lemma allowed Harris, in the ordinary polynomial case, to construct two types of particular cubature formulas. In Harris' paper, this lemma is an equivalence, although what is actually used to construct the cubature formulas is only this implication (see Sections 5 and 6 in \cite{harris}). It remains an open problem whether our generalization to the Laurent case can also allow the construction of particular L-Gaussian cubature formulas.

Some recent papers present alternative approaches to the construction of cubature formulas that are also of interest. For example, a characterization of cubature through Hankel operators is carried out in \cite{CoHu}. Algorithms to compute all minimal cubatures for a given domain and a given degree are elaborated there from moment theory. Also in \cite{Gau}, two new classes of stable high-order (with nonnegative weights) cubature rules are proposed by making use of a sufficiently large number of nodes (larger than the number of basis functions for which the cubature is exact) and that do not lie on an algebraic manifold generated by a polynomial of the same degree as the degree of accuracy of the cubature rule. This yields a linear system of equations and cubature weights are selected from the space of solutions minimizing certain norms related to the stability of the procedure.

\subsection{Multivariate strong moment problems}\label{Appl2}

In the univariate case, the order of a quadrature formula corresponds to matching some of the moments of the measure. The converse is also a problem of interest: given an infinite sequence of moments, is it possible to recover the underlying measure? Orthogonal polynomials and the analysis of the convergence of quadrature formulas are connected with the solutions of this problem. If a solution exists, it is of interest also to characterize all the possible measures that solve the moment problem (indeterminate case) and to obtain conditions such that the solution is unique (determinate case). There are several classical moment problems associated with polynomials (see \cite{Akh,Sho}): starting from a sequence $\{ \mu_k \}_{k\geq 0}$, to find conditions such that there exists a positive Borel measure $\mu$ satisfying $\int z^kd\mu(z)=\mu_k$, $k\in \NN\cup \{ 0 \}$. The most frecuently studied situations are those where the support of the measure is a compact interval (Hausdorff, see \cite{Hau}), the half-line (Stieltjes, see \cite{Sti}) and the whole real line (Hamburger, see \cite{Ham1,Ham3}).

An alternative to this problem is to use Laurent polynomials instead of ordinary polynomials, considering the sequence of moments $\{ \mu_k \}_{k\in \ZZ}$. The problem in this case is usually called {\em strong moment problem}, see \cite{BH,CDM2,CDM4,JTW,ON1,ON2}. This variant is linked to orthogonal Laurent polynomials and the construction of quadrature formulas exact in subspaces of Laurent polynomials, see e.g. these references.

The multidimensional moment problem has been also considered in the literature, see e.g. \cite{Berg,Fu,Havil1,Havil2,Xu2}. As indicated in \cite{Xu} (Section 3.2.3), the moment problem in several variables is much more difficult than its one-variable counterpart, and it is still not completely solved. However, some characterizations and sufficient conditions for a sequence to be determinate have been obtained (see e.g. these references for further details).

Due to the absence of a Theory of Orthogonal Laurent Polynomials in several variables up to now in the literature, {\em multivariate strong moment problems} have not been considered yet. In this respect, our results may represent a starting point to develop this theory, being key to this to have obtained a way to order the Laurent monomials in a suitable way, to have considered the multiplication plus inverse multiplication operator on each variable and to have deduced recurrence relations and the corresponding Favard theorem.

\section{Conclusions}\label{SecConc}

We have introduced for the first time in the literature the theory of sequences of orthogonal Laurent polynomials in two real variables $(x,y)$ (for the sake of simplicity, but it can be generalized to several variables) with respect to a positive Borel measure $\mu$ defined on $\RR^2$ such that $\{ x=0 \}\cup \{ y=0 \} \not\in \textrm{supp}(\mu)$. We have considered an appropriate ordering for the Laurent monomials $x^{i}y^{j}$, $i,j \in \ZZ$ that lets us obtain five term relations involving multiplication by $x+\frac{1}{x}$ and $y+\frac{1}{y}$. The corresponding matrices representations of these operators are block symmetric five diagonals. Our approach enables us to extend some known results for the ordinary polynomials to the Laurent case. In this respect, we have included a Favard's theorem and Christoffel-Darboux and confluent formulas. Also, a connection with the one variable case is done when the measure $\mu$ is a product measure of separate variables defined on the rectangle ${\mathcal R}=[a,b] \times [c,d]$, $0<a<b<\infty$, $0<c<d<\infty$.

In the one variable case, there are very few measures that give rise to explicit expressions for sequences of orthogonal Laurent polynomials. We could almost say that the only ones are practically the weight functions $\{ \omega_i \}_{i=1}^{6}$ mentioned at the end of Section 4. In general, these families are computed making use of Theorem \ref{recu1var}, under the knowledge of the corresponding moments. In the several variables case, there is not any sequence of orthogonal Laurent polynomials explicitly known, except in the situations described in Section 4. However, these families can always be obtained from the five-term relations obtained in Section 2 as long as the moments (\ref{moments}) exist and are computable. Also, unlike the situation in the one variable case, there are not known applications in the literature for the moment of orthogonal Laurent polynomials in two or more real variables. The introduction in this paper of the theory of orthogonal Laurent polynomials of two real variables allows us to extend results known for the univariate case to the multivariate. We have included a section with two applications for future research. In particular, cubature rules based on Laurent polynomials will be considered in a forthcoming paper.

\section{Acknowledgments}

The work of the first author is partially supported by IMAG-Mar\'ia de Maeztu grant CEX2020-001105-M.

\bibliography{biblio}

\begin{thebibliography}{10}

\bibitem{Akh}
N.~I. Akhiezer.
\newblock {\em The classical moment problem and some related questions in
  analysis}.
\newblock Hafner Publishing Co., New York, 1965.
\newblock Translated by N. Kemmer.

\bibitem{AKF26}
P.~Appell and J.~{Kamp\'e de F\'eriet}.
\newblock {\em Fonctions hyperg\'eom\'etriques et hypersph\'eriques.
  Polyn\^{o}mes d'Hermite.}
\newblock Gauthier-Villars, 1926.

\bibitem{AM22}
G.~Ariznabarreta and M.~Ma\~{n}as.
\newblock Multivariate orthogonal {L}aurent polynomials and integrable systems.
\newblock {\em Publ. Res. Inst. Math. Sci.}, 58(1):79--185, 2022.

\bibitem{Berg}
C.~Berg.
\newblock The multidimensional moment problem and semigroups.
\newblock {\em Moments in {M}athematics, {P}roceedings of {S}ymposia in
  {A}pplied {M}athematics, {A}merican {M}athemathical {S}ociety, {P}rovidence,
  RI.}, 37:110--124, 1987.

\bibitem{BH}
C.~Bonan-Hamada, W.~B. Jones, and W.~J. Thron.
\newblock Zeros of orthogonal {L}aurent polynomials and solutions of the strong
  {S}tieltjes moment problem.
\newblock {\em J. Comput. Appl. Math.}, 235(4):895--903, 2010.

\bibitem{But}
A.~Bultheel, C.~D\'{\i}az-Mendoza, P.~Gonz\'alez-Vera, and R.~Orive.
\newblock Orthogonal {L}aurent polynomials and quadrature formulas for
  unbounded intervals. {P}art {I}: {G}auss-type formulas.
\newblock {\em Rocky Mountain J. Math.}, 33(2):585--608, 2003.

\bibitem{CMV1}
M.~J. Cantero, L.~Moral, and L.~Vel{\'a}zquez.
\newblock Five-diagonal matrices and zeros of orthogonal polynomials on the
  unit circle.
\newblock {\em Linear Algebra Appl.}, 362:29--56, 2003.

\bibitem{CoHu}
M.~Collowald and E.~Hubert.
\newblock Algorithms for computing cubatures based on moment theory.
\newblock {\em Stud. Appl. Math.}, 141(4):501--546, 2018.

\bibitem{cools}
R.~Cools, I.~P. Mysovskikh, and H.~J. Schmid.
\newblock Cubature formulae and orthogonal polynomials.
\newblock {\em J. Comput. Appl. Math.}, 127:121--152, 2001.

\bibitem{SCop}
S.~C. Cooper, W.~B. Jones, and W.~J. Thron.
\newblock Orthogonal {L}aurent polynomials and continued fractions associated
  with log-normal distributions.
\newblock {\em J. Comput. Appl. Math.}, 32:39--46, 1990.

\bibitem{RCB1}
R.~Cruz-Barroso, L.~Daruis, P.~Gonz\'alez-Vera, and O.~Nj{\aa}stad.
\newblock Sequences of orthogonal {L}aurent polynomials, bi-orthogonality and
  quadrature formulas on the unit circle.
\newblock {\em J. Comput. Appl. Math.}, 206(2):950--966, 2007.

\bibitem{RCB2}
R.~Cruz-Barroso and S.~Delvaux.
\newblock Orthogonal {L}aurent polynomials on the unit circle and snake-shaped
  matrix factorizations.
\newblock {\em J. Approx. Theory}, 161(1):65--87, 2009.

\bibitem{RCB3}
R.~Cruz-Barroso, C.~D\'{\i}az~Mendoza, and R.~Orive.
\newblock Orthogonal {L}aurent polynomials. {A} new algebraic approach.
\newblock {\em J. Math. Anal. Appl.}, 408:40--54, 2013.

\bibitem{RCB5}
R.~Cruz-Barroso and P.~Gonz\'alez-Vera.
\newblock Orthogonal {L}aurent polynomials and quadrature on the unit circle
  and the real half-line.
\newblock {\em Electron. Trans. Numer. Anal.}, 19:113--134, 2005.

\bibitem{CDM1}
C.~D\'{\i}az-Mendoza, P.~Gonz\'alez-Vera, and M.~Jim\'enez~P\'aiz.
\newblock Orthogonal {L}aurent polynomials and two-point {P}ad\'e approximants
  associated with {D}awson's integral.
\newblock {\em J. Comput. Appl. Math.}, 179(1-2):195--213, 2005.

\bibitem{CDM2}
C.~D\'{\i}az-Mendoza, P.~Gonz\'alez-Vera, and M.~Jim\'enez-P\'aiz.
\newblock Strong {S}tieltjes distributions and orthogonal {L}aurent polynomials
  with applications to quadratures and {P}ad\'e approximation.
\newblock {\em Math. Comp.}, 74:1843--1870, 2005.

\bibitem{CDM4}
C.~D\'{\i}az-Mendoza, P.~Gonz\'alez-Vera, M.~Jim\'enez~P\'aiz, and
  F.~Cala-Rodr\'{\i}guez.
\newblock Orthogonal {L}aurent polynomials corresponding to certain strong
  {S}tieltjes distributions with applications to numerical quadratures.
\newblock {\em Math. Comp.}, 75(253):281--305, 2005.

\bibitem{CDM3}
C.~D\'{\i}az-Mendoza, P.~Gonz\'alez-Vera, M.~Jim\'enez~P\'aiz, and
  O.~Nj{\aa}stad.
\newblock Orthogonality and recurrence for ordered {L}aurent polynomial
  sequences.
\newblock {\em J. Comput. Appl. Math.}, 235(4):982--997, 2010.

\bibitem{Xu}
C.F. Dunkl and Y.~Xu.
\newblock {\em Orthogonal polynomials of several variables}.
\newblock Cambridge University Press, 2014.

\bibitem{Fu}
B.~Fuglede.
\newblock The multidimensional moment problem.
\newblock {\em Exposition. Math.}, 1:47--65, 1983.

\bibitem{Gau}
J.~Glaubitz.
\newblock Stable high-order cubature formulas for experimental data.
\newblock {\em J. Comput. Phys.}, 447:Paper No. 110693, 22, 2021.

\bibitem{Ham1}
H.~Hamburger.
\newblock Ueber eine {E}rweiterung des {S}tietjesschen {M}oment {P}roblems {I}.
\newblock {\em Math. Ann.}, 81:235--319, 1920.

\bibitem{Ham3}
H.~Hamburger.
\newblock Ueber eine {E}rweiterung des {S}tietjesschen {M}oment {P}roblems
  {III}.
\newblock {\em Math. Ann.}, 82:168--187, 1921.

\bibitem{harris}
L.~A. Harris.
\newblock Lagrange polynomials, reproducing kernels and cubature in two
  dimensions.
\newblock {\em J. Approx. Theory}, 195:43--56, 2015.

\bibitem{Hau}
F.~Hausdorff.
\newblock Momentprobleme f{\"u}r ein endliches {I}ntervall.
\newblock {\em Math. Z.}, 16(1):220--248, 1923.

\bibitem{Havil1}
E.~K. Haviland.
\newblock On the {M}omentum {P}roblem for {D}istribution {F}unctions in {M}ore
  than {O}ne {D}imension.
\newblock {\em Amer. J. Math.}, 57:562--568, 1935.

\bibitem{Havil2}
E.~K. Haviland.
\newblock On the {M}omentum {P}roblem for {D}istribution {F}unctions in {M}ore
  {T}han {O}ne {D}imension. {II}.
\newblock {\em Amer. J. Math.}, 58:164--168, 1936.

\bibitem{HvR}
E.~Hendriksen and H.~{van Rossum}.
\newblock Orthogonal {L}aurent polynomials.
\newblock {\em Nederl. Akad. Wetensch. Indag. Math.}, 48(1):17--36, 1986.

\bibitem{Her65}
C.~Hermite.
\newblock Sur quelques d\'eveloppement en s\'erie de fonctions de plusieurs
  variables.
\newblock {\em Comptes Rendus Acad. Sci. Paris}, 235(60):370--377, 432--440,
  461--466, 512--518, 1865.

\bibitem{Horn}
R.~A. Horn and C.~R. Johnson.
\newblock {\em Matrix Analysis}.
\newblock Cambridge University Press, 2013.

\bibitem{JO1}
W.~B. Jones and O.~Nj{\aa}stad.
\newblock Orthogonal {L}aurent polynomials and strong moment theory: a survey.
\newblock {\em J. Comput. Appl. Math.}, 105(1-2):51--91, 1999.

\bibitem{JT1}
W.~B. Jones and W.~J. Thron.
\newblock Orthogonal {L}aurent polynomials and {G}aussian quadrature.
\newblock {\em in: K.E. Gustafson, W.P. Reinhardt (Eds.), Quantum Mechanics in
  Mathematics, Chemistry and Physics, Plenum Publishing Company, New York},
  pages 449--455, 1981.

\bibitem{JNT1}
W.~B. Jones, W.~J. Thron, and O.~Nj{\aa}stad.
\newblock Orthogonal {L}aurent polynomials and the strong {H}amburger moment
  problem.
\newblock {\em J. Math. Anal. Appl.}, 94:528--554, 1984.

\bibitem{JTW}
W.~B. Jones, W.~J. Thron, and H.~Waadeland.
\newblock A strong {S}tieltjes moment problem.
\newblock {\em Trans. Amer. Math. Soc.}, 261:503--528, 1980.

\bibitem{Ko75}
T.~H. Koornwinder.
\newblock Two-variable analogues of the classical orthogonal polynomials.
\newblock In Richard~A. Askey, editor, {\em Theory and Application of Special
  Functions}, pages 435--495. Academic Press, 1975.

\bibitem{Ko82b}
M.~A. Kowalski.
\newblock Orthogonality and recursion formulas for polynomials in n variables.
\newblock {\em SIAM J. Math. Anal.}, 13:316--323, 1982.

\bibitem{Ko82a}
M.~A. Kowalski.
\newblock The recursion formulas for orthogonal polynomials in n variables.
\newblock {\em SIAM J. Math. Anal.}, 13:309--315, 1982.

\bibitem{KS67}
H.~L. Krall and I.~M. Sheffer.
\newblock Orthogonal polynomials in two variables.
\newblock {\em Ann. Mat. Pura Appl. Serie 4}, 76:325--376, 1967.

\bibitem{lasse}
J.~B. Lasserre.
\newblock The existence of {G}aussian cubature formulas.
\newblock {\em J. Approx. Theory}, 164:572--585, 2012.

\bibitem{Mys}
I.~P. Mysovskikh.
\newblock A multidimensional analog of quadrature formula of gaussian type and
  the generalized problem of radon (in russian).
\newblock {\em Vopr. Vychisl. i Prikl. Mat.}, 38:55--69, 1970.

\bibitem{ON1}
O.~Nj{\aa}stad.
\newblock Extremal solutions of the strong {S}tieltjes moment problem.
\newblock {\em J. Comput. Appl. Math.}, 65:309--318, 1995.

\bibitem{ON2}
O.~Nj{\aa}stad.
\newblock Solutions of the strong {H}amburger moment problem.
\newblock {\em J. Math. Anal. Appl.}, 197:227--248, 1996.

\bibitem{OT1}
O.~Nj{\aa}stad and W.~J. Thron.
\newblock The theory of sequences of orthogonal {L}-polynomials, {P}ad\'e
  approximants and continued fractions.
\newblock {\em in: H. Waadeland, H. Wallin (Eds.), Det Kongelige Norsk
  Videnskabers Selskab}, 1:54--91, 1983.

\bibitem{OT2}
O.~Nj{\aa}stad and W.~J. Thron.
\newblock Unique solvability of the strong {H}amburger moment problem.
\newblock {\em J. Austral. Math. Soc. (Series A)}, 40:5--19, 1986.

\bibitem{Sho}
J.~A. Shohat and J.~D. Tamarkin.
\newblock {\em The {P}roblem of {M}oments}, volume Vol. I of {\em American
  Mathematical Society Mathematical Surveys}.
\newblock American Mathematical Society, New York, 1943.

\bibitem{Simon}
B.~Simon.
\newblock {\em Orthogonal polynomials on the unit circle}, volume~54 of {\em
  Colloquium Publications}.
\newblock AMS, 2005.

\bibitem{ranga1}
A.~Sri~Ranga.
\newblock Another quadrature rule of highest algebraic degree of precision.
\newblock {\em Numer. Math.}, 68:283--294, 1994.

\bibitem{ranga2}
A.~Sri~Ranga, E.~X.~L. {de Andrade}, and J.~H. McCabe.
\newblock Some consequences of a symmetry in strong distributions.
\newblock {\em J. Math. Anal. Appl.}, 193(1):158--168, 1995.

\bibitem{ranga3}
A.~Sri~Ranga and J.~H. McCabe.
\newblock On the extensions of some classical distributions.
\newblock {\em Proc. Edinb. Math. Soc., II. Ser.}, 34(1):19--29, 1991.

\bibitem{ranga4}
A.~Sri~Ranga and J.~H. McCabe.
\newblock On pairwise related strong stieltjes distributions.
\newblock {\em Skr. K. Nor. Vidensk. Selsk.}, 3:12 pp, 1996.

\bibitem{Sti}
T.~J. Stieltjes.
\newblock Recherches sur les fractions continues.
\newblock {\em Ann. Fac. Sci. Toulouse Math. (6)}, 4(2):Ji, J36--J75, 1995.
\newblock Reprint of the 1894 original.

\bibitem{Su99}
P.~K. Suetin.
\newblock {\em Orthogonal polynomials in two variables}, volume~3 of {\em
  Analytical Methods and Special Functions}.
\newblock Gordon and Breach Science Publishers, Amsterdam, 1999.

\bibitem{Thr}
W.~J. Thron.
\newblock {L}-polynomials orthogonal on the unit circle.
\newblock {\em in: A. Cuyt editor, Nonlinear Methods and Rational
  Approximation. Reidel Publishing Company, Dordrecht}, pages 271--278, 1988.

\bibitem{Xu2}
Y.~Xu.
\newblock Recurrence formulas for multivariate orthogonal polynomials.
\newblock {\em Math. Comp.}, 62(206):687--702, 1994.

\bibitem{Zer34}
F.~Zernike.
\newblock Beugungstheorie des schneidenver-fahrens und seiner verbesserten
  form, der phasenkontrastmethode.
\newblock {\em Physica}, 1(7):689--704, 1934.

\end{thebibliography}
\bibliographystyle{plain}

\end{document}